\documentclass[twoside,11pt]{article}

\usepackage[preprint]{jmlr2e}

\usepackage{dsfont}
\usepackage{amsmath}
\usepackage{url}
\usepackage{hyperref}
\usepackage{color}
\definecolor{navy}{rgb}{0,0,0.5}
\hypersetup{
    colorlinks=true,
    citecolor=navy,
    linkcolor=blue,
    filecolor=magenta,
    urlcolor=cyan,
}

\let\emptyset\varnothing

\makeatletter
\newcommand*\rel@kern[1]{\kern#1\dimexpr\macc@kerna}
\newcommand*\widebar[1]{%
  \begingroup
  \def\mathaccent##1##2{%
    \rel@kern{0.8}%
    \overline{\rel@kern{-0.8}\macc@nucleus\rel@kern{0.2}}%
    \rel@kern{-0.2}%
  }%
  \macc@depth\@ne
  \let\math@bgroup\@empty \let\math@egroup\macc@set@skewchar
  \mathsurround\z@ \frozen@everymath{\mathgroup\macc@group\relax}%
  \macc@set@skewchar\relax
  \let\mathaccentV\macc@nested@a
  \macc@nested@a\relax111{#1}%
  \endgroup
}
\makeatother

\newtheorem{assumption}{Assumption}

\usepackage{color}

\newcommand{\R}{\mathbf{R}}
\newcommand{\minimize}{\operatorname*{minimize}}
\newcommand{\maximize}{\operatorname*{maximize}}

\newcommand{\argmin}{\operatorname*{arg\,min}}

\newcommand{\dom}{\operatorname{dom}}
\newcommand{\rint}{\operatorname{rint}}
\newcommand{\E}{\operatorname*{\mathbf{E}}}
\newcommand{\prox}{\operatorname{\mathbf{prox}}}

\newcommand{\Proj}{\mathbf{proj}}
\newcommand{\Prob}{\mathrm{Pr}}

\newcommand{\ones}{\mathbf{1}}
\newcommand{\onebb}{\mathds{1}}
\newcommand{\nablas}{\nabla_{\!s}}
\newcommand{\llangle}{\left\langle}
\newcommand{\rrangle}{\right\rangle}

\newcommand{\V}{ V}
\newcommand{\cS}{\mathcal{S}}
\newcommand{\cA}{\mathcal{A}}

\newcommand{\cC}{\mathcal{C}}

\newcommand{\cT}{\mathcal{T}}
\newcommand{\dS}{{|\cS|}}
\newcommand{\dA}{{|\cA|}}
\newcommand{\pistar}{\pi^{\star}}
\newcommand{\pizero}{\pi^{(0)}}
\newcommand{\pii}{\pi^{(i)}}
\newcommand{\piK}{\pi^{(K)}}
\newcommand{\pik}{\pi^{(k)}}
\newcommand{\pikp}{\pi^{(k+1)}}

\newcommand{\xzero}{ x^{(0)}}
\newcommand{\xk}{ x^{(k)}}
\newcommand{\xkp}{ x^{(k+1)}}

\newcommand{\zk}{ z^{(k)}}

\newcommand{\bQ}{\widebar{Q}}
\newcommand{\hQ}{\widehat{Q}}
\newcommand{\Dk}{ D^{\star}_{k}}
\newcommand{\Dkp}{ D^{\star}_{k+1}}
\newcommand{\Dzero}{ D^{\star}_{0}}
\newcommand{\Cstar}{ C^{\star}_{\rho}}

\jmlrheading{1}{2000}{1-48}{4/00}{10/00}{xiao2021}{Lin Xiao}

\firstpageno{1}

\begin{document}

\title{On the Convergence Rates of Policy Gradient Methods}

\author{\name Lin Xiao \email linx@fb.com\\
       \addr Meta AI Research \\
       Seattle, WA 98109, USA}

\editor{TBD}

\maketitle

\begin{abstract}
We consider infinite-horizon discounted Markov decision problems with finite state and action spaces and study the convergence rates of the projected policy gradient method and a general class of policy mirror descent methods, all with direct parametrization in the policy space. 
First, we develop a theory of weak gradient-mapping dominance and use it to prove sharper sublinear convergence rate of the projected policy gradient method. 
Then we show that with geometrically increasing step sizes, a general class of policy mirror descent methods, including the natural policy gradient method and a projected Q-descent method, all enjoy a linear rate of convergence without relying on entropy or other strongly convex regularization.
Finally, we also analyze the convergence rate of an inexact policy mirror descent method and estimate its sample complexity under a simple generative model.
\end{abstract}

\bigskip

\smallskip

\section{Introduction}
\label{sec:intro}

Markov decision process (MDP) is a fundamental model for sequential decision-making.
In this paper, we consider infinite-horizon, discounted Markov decision problems (DMDPs) with finite state and action spaces. 
They are specified as a 5-tuple $(\cS,\cA, P, R, \gamma)$, where
$\cS$ is a finite state space with cardinality $\dS$,
$\cA$ is a finite action space with cardinality $\dA$,
$P$~is a transition probability function with $P(s'|s,a)$ denoting the probability of transitioning to~$s'$ when taking action~$a$ from state~$s$,
$R:\cS\times\cA\to [0,1]$ is a reward function with $R_{s,a}$ or $R(s,a)$ being the (expected) reward of taking action~$a$ from state~$s$,
and finally $\gamma\in[0,1)$ is a discount factor applied to the reward one-step in the future.

Starting from an initial state $s_0\in\cS$, an agent takes an action $a_t\in\cA$ at each time step $t=0,1,2,\ldots$, which leads to the next state $s_{t+1}$ with probability $P(s_{t+1}|s_t, a_t)$,
and obtains the immediate reward $r_t=R(s_t, a_t)$.
Such interactions generate a \emph{trajectory}
\[
(s_0, a_0, r_0),\, (s_1, a_1, r_1),\, (s_2, a_2, r_2), \, \ldots .
\]
The goal of the agent is to find a \emph{policy} of choosing the actions $a_0,a_1,a_2,\ldots$ that maximizes the discounted cumulative reward
$\E\bigl[\sum_{t=0}^\infty \gamma^t r_t\bigr]$.
Here the expectation is taken with respect to the possible randomness in $s_0$, any randomness in choosing the actions $a_t$, and the randomness of state transitions prescribed by~$P$.

In general, a policy that determines the action at time~$t$ may depends on the whole history of the trajectory up to time~$t$. 
A \emph{stationary policy} $\pi$ specifies a decision rule that depends only on the current state. Specifically, we let $\pi_s\in\Delta(\cA)$ be the decision rule at state~$s$, where $\Delta(\cA)$ denotes the probability simplex supported on~$\cA$, 
and $\pi_{s,a}$ denotes the probability of taking action~$a$ at state~$s$.
The value of a stationary policy $\pi\in\Delta(\cA)^\dS$ starting from an arbitrary state~$s$ is defined as 
\begin{equation}\label{eqn:v-s-def}
\V_s(\pi):=\E
\left[\sum_{t=0}^\infty \gamma^t R(s_t,a_t) \,\Big|\, s_0=s\right],
\end{equation}
where the expectation is taken with respect to $a_t\sim\pi_{s_t}$ and $s_{t+1}\sim P(\cdot|s_t,a_t)$ for all $t\geq 0$. 
We define $\V:\Delta(\cA)^\dS\to\R^\dS$ as a vector-valued function with components $\V_s(\pi)$. 
By the assumption that $R(s,a)\in[0,1]$ for all $(s,a)\in\cS\times\cA$, we immediately have 
\begin{equation}\label{eqn:value-bound}
0\leq \V_s(\pi)\leq\sum_{t=0}^{\infty}\gamma^t =\frac{1}{1-\gamma},
\qquad \forall\,s\in\cS,\quad\forall\,\pi\in\Delta(\cA)^\dS.
\end{equation}

The conventional formulation of DMDP is about maximizing the discounted total reward. 
In this paper, we adopt a minimization formulation in order to better align with conventions in the optimization literature.
To this end, we regard each $R(s,a)\in[0,1]$ as a value measuring \emph{regret} rather than reward.
Given a reward matrix~$R$, we can reset $R(s,a)\gets 1-R(s,a)$ for all $(s,a)\in\cS\times\cA$ to turn it into a regret matrix.
Suppose $\rho\in\Delta(\cS)$ is an arbitrary initial state distribution.
We consider the problem of minimizing 
\begin{equation}\label{eqn:v-rho-def}
\V_\rho(\pi):=\E_{s\sim\rho}\V_s(\pi)
=\E\left[\sum_{t=0}^\infty \gamma^t R(s_t,a_t) \,\Big|\, s_0\sim\rho\right].
\end{equation}
For infinite-horizon DMDPs with finite state and actions spaces, there exists a (deterministic) stationary policy $\pi^\star$ that is simultaneously optimal in minimizing $\V_s(\cdot)$ for all $s\in\cS$
\citep[e.g.,][Section~6.2.4]{Puterman1994book}. 
Such a solution is insensitive to the choice of $\rho$.

In this paper, we focus on \emph{policy gradient methods} for minimizing the weighted value function $\V_\rho$. These methods generate a sequence of policies $\{\pik\}$ through repeated evaluation of the policy gradient $\nabla \V_\mu$, where $\mu\in\Delta(\cS)$ is not necessarily equal to~$\rho$.
The most straightforward variant is the projected policy gradient method,
\begin{equation}\label{eqn:ppg-intro}
\pikp = \Proj_{\Pi}\left(\pik - \eta_k\nabla \V_{\mu}(\pik)\right),
\end{equation}
where $\eta_k$ is the step size, $\Pi:= \Delta(\cA)^{\dS}$ is the set of feasible policies, and $\Proj_\Pi(\cdot)$ denotes projection onto~$\Pi$ in the Euclidean norm.
More generally, policy gradient methods can be derived from the mirror-descent form
\begin{equation}\label{eqn:pmd-intro}
\pikp = \argmin_{\pi\in\Pi} \Bigl\{
\eta_k \bigl\langle \nabla \V_{\mu}(\pik),\,\pi\bigr\rangle
+ D_k(\pi,\pik) \Bigr\},
\end{equation}
where $D_k(\cdot,\cdot)$ is a distance-like function that may depend on $\pik$.
For example, setting $D_k$ as the squared Euclidean distance yields 
the projected policy gradient method~\eqref{eqn:ppg-intro}.
\citet{Shani2020aaai} showed that by setting $D_k$ as an appropriately weighted Kullback-Leibler (KL) divergence, one recovers the natural policy gradient (NPG) method of \citet{Kakade2001NPG}.
In general, we can think of~\eqref{eqn:pmd-intro} as a class of preconditioned policy gradient methods.
The main results of this paper concern the convergence rates of such methods.

\subsection{Previous Work}
\label{sec:previous-work}

Many classical algorithms for DMDP are based on dynamical programming \citep{Bellman1957book}, including value iteration, policy iteration, temporal difference learning and Q-learning
\citep[see, e.g.,][]{Puterman1994book,BertsekasTsitsiklis1996book,SuttonBarto2018}.
Analyses of these methods in the tabular case mostly rely on the contraction property of the Bellman operator, which are difficult to extend with nonlinear function approximation and policy parametrization.  
In contrast, policy gradient methods \citep{Williams1992,Sutton2000PolicyGrad,KondaTsitsiklis2000,Kakade2001NPG}
aim to find a local minimum of an expected value function,
thus are applicable to any differentiable policy parametrization
and admit easy extensions to function approximation.
In particular, they appear to work well when parametrized with modern deep neural networks \citep{Schulman2015TRPO,Schulman2017PPO}.

Despite the long history and empirical successes of policy gradient methods, their convergence properties are not well understood until recently.
For example, it was widely accepted that they converge asymptotically to a stationary point or a local minimum because the objective function is nonconvex in general. 
However, \citet{Fazel2018LQR} show that for linear quadratic control problems, policy gradient methods converge to the global optimal solution despite the nonconvex cost function, thanks to a gradient dominance property \citep{Polyak1963}. 
\citet{Agarwal2021jmlr} derive a variational gradient-dominance property
and use it to obtain global convergence of the projected policy gradient method~\eqref{eqn:ppg-intro}.
\citet{BhandariRusso2019global} identify more general structural properties of policy gradient methods to ensure gradient domination and hence convergence to global optimum.

Using direct policy parametrization (over $\pi\in\Delta(\cA)^\dS$),
\citet{Agarwal2021jmlr} show that the projected policy gradient method~\eqref{eqn:ppg-intro} converges to a global optimum at an $O(1/\sqrt{k})$ sublinear rate.
Specifically, the number of iterations to obtain $\V_\rho(\pik)-\V_\rho^\star\leq\epsilon$ is
\begin{equation}\label{eqn:ppg-rate-intro}
O\biggl(\frac{\dS\dA}{(1-\gamma)^6\,\epsilon^2}\left\|\frac{d_\rho(\pi^\star)}{\mu}\right\|_\infty^2\biggr),
\end{equation}
where $d_\rho(\pistar)\in\Delta(\cS)$ is a \emph{discounted state-visitation distribution}
and $\bigl\|d_\rho(\pistar)/\mu\bigr\|_\infty$ is a \emph{distribution mismatch coefficient} (see Section~\ref{sec:dmc} for definition and explanation).
\citet{Junyu2020NeurIPS} develop a variational policy gradient framework and use it to show that the projected policy gradient method converges to global optimum at a faster $O(1/k)$ rate.
In both cases, the constants in the iteration complexity are very large and depend on the Lipschitz constant characterizing the smoothness of the objective function. 

\citet{Shani2020aaai} show that the natural policy gradient (NPG) method
\citep{Kakade2001NPG} can be cast as a special case of policy mirror descent method~\eqref{eqn:pmd-intro} and has an $O(1/\sqrt{k})$ convergence rate.
\citet{Agarwal2021jmlr} improve the convergence rate of NPG to $O(1/k)$; more concretely, the number of iterations to obtain $\V_\rho(\pik)-\V_\rho^\star\leq\epsilon$ is
\begin{equation}\label{eqn:npg-sublinear}
\frac{2}{(1-\gamma)^2\epsilon},
\end{equation}
which is independent of the dimensions $\dS$ and $\dA$ or any distribution mismatch coefficient.
Interestingly, the step sizes that guarantee such a rate can be chosen arbitrarily large, regardless of the Lipschitz constant of the policy gradient. 

With entropy regularization (added to the DMDP objective), \citet{Cen2020NPGentropy} show that the NPG method has linear (geometric) convergence. 
Their approach rely on the contraction property of a generalized Bellman operator and the convergence guarantees are in terms of the infinity norm of the ``soft'' $Q$-functions. With appropriate choice of the regularization parameter and step size, they obtain iteration complexity on the order of
\begin{equation}\label{eqn:npg-linear}
\frac{1}{1-\gamma}\log\frac{1}{(1-\gamma)\epsilon}.
\end{equation}
\citet{Lan2021pmd} proposes a general policy mirror descent method that is similar to~\eqref{eqn:pmd-intro} with either convex or strongly convex regularizations.
He focuses on the case of minimizing $\V_{\rho^\star}$ where $\rho^\star$ is the stationary distribution of the MDP under the optimal policy $\pistar$, which avoids any distribution mismatch coefficient in the analysis.
In order to guarantee
$\V_{\rho^\star}(\pik)-\V_{\rho^\star}^\star\leq\epsilon$, 
\citet{Lan2021pmd} obtains iteration complexity on the orders of~\eqref{eqn:npg-sublinear} and~\eqref{eqn:npg-linear} for the settings without and with entropy regularization, respectively.
More interestingly, \citet{Lan2021pmd} also obtained linear convergence for the un-regularized DMDP using diminishing regularization combined with increasing step sizes (while maintaining a constant product of the two).

More recently, 
\citet{Zhan2021Regularized} extend the framework of \citet{Lan2021pmd} to accommodate a broader class of convex regularizers including those that are nonsmooth.
For un-regularized DMDP, 
\citet{Khodadadian2021} show that the NPG method can obtain linear convergence with an adaptive step-size rule, and
\citet{BhandariRusso2021linear} show that several variants of policy gradient methods has linear convergence with exact line search.

For the exact policy gradient method with softmax parametrization, \citet{Agarwal2021jmlr} show that it converges asymptotically to a global optimum, and attains an $O(1/\sqrt{k})$ rate with log barrier regularization.
\citet{Mei2020icml} derive an $O(1/k)$ convergence rate 
and \citet{Mei2021icml} further improve it to linear convergence by exploiting non-uniform variants of the smoothness and gradient dominance properties. However, these fast rates are associated with problem-dependent constants that can be very large \citep{LiWeiChiGuChen2021}.

\subsection{Contributions and Outline}

In this paper, we present a systematic study of policy gradient methods with direct policy parametrization, focusing on their convergence rates for minimizing $\V_\rho$ over $\pi\in\Delta(\cA)^\dS$.

Section~\ref{sec:dmdp} contains an overview of structural properties of DMDP that are well-known but essential for the main results of the paper.

In Section~\ref{sec:ppg}, we develop a theory of \emph{weak gradient-mapping domination} for general nonconvex composite optimization, and use it to obtain an $O(1/k)$ convergence rate for the projected policy gradient method. Concretely, our result on iteration complexity replaces $\epsilon^{-2}$ in~\eqref{eqn:ppg-rate-intro} with $\epsilon^{-1}$ and $(1-\gamma)^{-6}$ with $(1-\gamma)^{-5}$.
Although this result is the same as the one obtained by \citet{Junyu2020NeurIPS}, our analysis are quite different.
\citet{Junyu2020NeurIPS} exploit the bijection structure of the primal-dual DMDP formulations, while we derive this result as a special case of nonconvex optimization with weak gradient-mapping domination which, to our best knowledge, is new and of independent interest.

In Section~\ref{sec:pmd}, we study exact policy mirror descent methods of the form~\eqref{eqn:pmd-intro}.
First, we show that with a constant step size (which can be arbitrarily large), they obtain the same dimension-free iteration complexity~\eqref{eqn:npg-sublinear}. This result extend the one of \citet{Agarwal2021jmlr} on NPG with KL-divergence to a general class of Bregman divergences, including a projected $Q$-descent method derived with squared Euclidean distance.
Second, we show that with geometrically increasing step sizes, as simple as $\eta_{k+1}=\eta_k/\gamma$, policy mirror descent methods enjoy linear convergence \emph{without} relying on any regularization. Specifically, their iteration complexity for reaching 
$\V_\rho(\pik)-\V_\rho^\star\leq\epsilon$ is 
\[
\frac{1}{1-\gamma}\left\|\frac{d_\rho(\pistar)}{\rho}\right\|_\infty
\log\frac{2}{(1-\gamma)\epsilon} .
\]
If $\rho$ is set to be the stationary distribution under the optimal policy~$\pistar$,
then the distribution mismatch coefficient $\bigl\|d_\rho(\pistar)/\rho\bigr\|_\infty=1$ and we recover~\eqref{eqn:npg-linear}. 
In addition, we discuss conditions for superlinear convergence and make connections with the classical Policy Iteration method.

In Section~\ref{sec:inexact-pmd}, we investigate the iteration complexity of inexact policy mirror descent methods and show that the geometrically increasing step sizes do not cause instability even with errors in evaluating the policy gradients or $Q$-functions.
They converge with the same linear rate up to an asymptotic error floor.
With a simple $Q$-estimator by repeated simulation of truncated trajectories, we obtain a sample complexity of
\[
\widetilde{O}\left(\frac{\dS\dA}{(1-\gamma)^8\,\epsilon^2}
\left\|\frac{d_\rho(\pistar)}{\rho}\right\|_\infty^3 \right),
\]
where the notation $\widetilde{O}(\cdot)$ hides poly-logarithmic factors of $\dS\dA$, $1/(1-\gamma)$ and $1/\epsilon$.

Finally, in Section~\ref{sec:conclusion}, we discuss the limitations of our work and possible extensions.

\section{Preliminaries on DMDP}
\label{sec:dmdp}

In this section, we overview the structural properties of DMDP that are essential for the developments in later sections.
We start with a few definitions.
Let $\Delta(\cS)$ denote the probability simplex defined over the state space~$\cS$, i.e., 
\[
\Delta(\cS) = \Bigl\{ p\in\R^\dS~\big|~\textstyle\sum_{s\in\cS}p_s=1,~p_s\geq 0~\textrm{for all}~ s\in\cS\Bigr\}.
\]
Similarly, $\Delta(\cA)$ denotes the probability simplex over the action space~$\cA$.
The set of admissible policies for DMDP is defined as
\[
\Pi := \Delta(\cA)^{\dS} = \left\{\pi=\{\pi_s\}_{s\in\cS}~\big|~
\pi_s\in\Delta(\cA) ~\mbox{for all}~ s\in\cS\right\}.
\]
With slight abuse of notation, we define the following functions of~$\pi\in\Pi$: 
\begin{itemize}
\item $P:\Pi\to\R^{\dS\times\dS}$: a matrix function with entries $P_{s,s'}(\pi) = \sum_{a\in\cA}\pi_{s,a}P(s'|s,a)$;
\item $r:\Pi\to \R^\dS$: a vector function with components $r_s(\pi)=\sum_{a\in\cA}\pi_{s,a}R_{s,a}$.
\end{itemize}
%

Using the definitions above, 
the value function $\V:\Pi\to\R^\dS$, whose components $\V_s$ are defined in~\eqref{eqn:v-s-def}, admits the following analytic form
\citep[see, e.g.,][Section~6.1]{Puterman1994book}
\begin{equation}\label{eqn:v-def-P-r}
\V(\pi) = \sum_{t=0}^\infty\gamma^t P(\pi)^t\, r(\pi) = \bigl(I-\gamma P(\pi)\bigr)^{-1}r(\pi).
\end{equation}
Since $P(\pi)$ is a row stochastic matrix and $0\leq\gamma<1$, 
the spectral norm of $\gamma P(\pi)$ is strictly less than one (by the Perron-Frobenius theorem) and thus $I-\gamma P(\pi)$ is always invertible. 
%
Given $\rho\in\Delta(\cS)$, the weighted value function $\V_\rho$ defined in~\eqref{eqn:v-rho-def} can be written as
\begin{equation}\label{eqn:v-rho-analytic}
\V_\rho(\pi)
=\rho^T \V(\pi) 
=\rho^T(I-\gamma P(\pi))^{-1}r(\pi) .
\end{equation}
Here we treat~$\rho$ as a column vector and
use matrix multiplication conventions. 

\subsection{Distribution Mismatch Coefficient}
\label{sec:dmc}

Starting from $s\in\cS$, the discounted state-visitation distribution under a policy $\pi$ is a vector $d_s(\pi)\in\Delta(\cS)$ whose components are defined as 
\begin{equation}\label{eqn:dsv-def}
d_{s,s'}(\pi):=(1-\gamma)\sum_{t=0}^\infty\gamma^t\,\Prob^{\pi}\bigl(s_t=s'\,|\,s_0=s\bigr),
\qquad \forall\,s'\in\cS.
\end{equation}
The coefficient $1-\gamma$ ensures that $\sum_{s'\in\cS}d_{s,s'}(\pi)=1$. 
In fact, $d_{s,s'}(\pi)$ is the $(s,s')$ entry of the matrix $(1-\gamma)(I-\gamma P(\pi))^{-1}$.
In other words, if we define $e_s\in\R^\dS$ with components $e_{s,s'}=1$ if $s=s'$ and $0$ otherwise, then we have
\begin{equation}\label{eqn:dsv-matrix}
d_{s,s'}(\pi) = (1-\gamma) e_s^T \bigl(I-\gamma P(\pi)\bigr)^{-1} e_{s'},
\qquad \forall\,s,s'\in\cS.
\end{equation}
Given an initial state distribution $\rho\in\Delta(\cS)$, we define 
$d_\rho(\pi)\in\Delta(\cS)$ with components
\[
d_{\rho,s'}(\pi) =\E_{s\sim\rho}d_{s,s'}(\pi) = \sum_{s\in\cS}\rho_s d_{s,s'}(\pi) .
\]
Some useful facts from the above definitions are:
\begin{equation}\label{eqn:dsv-lower-bound}
d_{s,s}(\pi) \geq 1-\gamma\qquad\mbox{and}\qquad 
d_{\rho,s}(\pi)\geq(1-\gamma)\rho_s,\qquad \forall\,s\in\cS.
\end{equation}

For any $\rho,\mu\in\Delta(\cS)$, we define the distribution mismatch of~$\rho$ from~$\mu$ as
\[
\left\|\frac{\rho}{\mu}\right\|_\infty
:=\max_{s\in\cS}\frac{\rho_s}{\mu_s},
\]
with the convention $0/0=1$.
This is an asymmetric measure of mismatch and
it is finite if and only if the support (set of indices with nonzero entries) of~$\mu$ contains that of~$\rho$.
If $\mu$ is the uniform distribution, then the mismatch is bounded by $\dS$.

The convergence properties of policy gradient methods often depend on the distribution mismatch coefficients between two discounted state-visitation distributions \citep[e.g.,][]{Kakade2002icml,Agarwal2021jmlr}.
According to~\eqref{eqn:dsv-lower-bound}, we have
for any $\rho,\mu\in\Delta(\cS)$ and $\pi,\pi'\in\Pi$,
\begin{equation}\label{eqn:dmc-ub}
\left\|\frac{d_\rho(\pi)}{d_\mu(\pi')}\right\|_\infty
\leq\frac{1}{1-\gamma}\left\|\frac{d_\rho(\pi)}{\mu}\right\|_\infty.
\end{equation}
Our results in this paper mostly concern the case with $\rho=\mu$ and $\pi=\pistar$, i.e., 
\[
\Cstar := \left\|\frac{d_\rho(\pistar)}{\rho}\right\|_\infty.
\]
In order for $\Cstar$ to be finite, it suffices to assume $\rho>0$, which means $\rho_s>0$ for all $s\in\cS$.

The distribution mismatch coefficient is closely related to 
the \emph{concentrability coefficients} in the analysis of approximate dynamic programming algorithms \citep{Munos2003,Munos2005,SzepesvariMunos2008}.
In fact, $C_\rho^\star$ is considered the ``best'' one among all concentrability coefficients in the sense that 
it does not impose any restrictions on the MDP dynamics
and it can be finite when other concentrability coefficients are infinite \citep{Scherrer2014}.
See \citet[][Section~2]{Agarwal2021jmlr} for further discussions.

If $\rho$ is chosen as the stationary distribution of the MDP under the optimal policy $\pistar$, denoted as $\rho^\star$, then we have 
$d_{\rho^\star}(\pistar)=\rho^\star$ and hence $C^\star_{\rho^\star}=1$.
This is the setting adopted by \citet{LiuCaiYangWang2019NeurIPS} and
\citet{Lan2021pmd}, which leads to simplified analysis for minimizing $\V_{\rho^\star}$. 
For DMDP with entropy regularization \citep{Lan2021pmd,Cen2020NPGentropy}, the resulting $\rho^\star$ always have full support over~$\cS$.
However, in general $\rho^\star$ may not have full support over~$\cS$ unless the underlying MDP is \emph{ergodic} \citep[Section~A.2]{Puterman1994book}.

\subsection{$Q$-functions, Policy Gradient and Performance Difference Lemma}

For each pair $(s,a)\in\cS\times\cA$,
the state-action value function $Q_{s,a}:\Pi\to\R$ is defined as
\begin{equation}\label{eqn:Q-def}
Q_{s,a}(\pi) :=\E
\left[\sum_{t=0}^\infty \gamma^t R(s_t,a_t) \,\Big|\, s_0=s,\,a_0=a\right],
\end{equation}
where the expectation is taken with respect to $a_t\sim\pi_{s_t}$ and $s_{t+1}\sim P(\cdot|s_t,a_t)$ for all $t\geq 0$. 
It is straightforward to verify that 
\begin{equation}\label{eqn:Q-v-relation}
Q_{s,a}(\pi) = R_{s,a} + \gamma\sum_{s'\in\cS}P(s'|s,a) \V_{s'}(\pi).
\end{equation}
Let $Q_s(\pi)\in\R^\dA$ denote the vector with components $Q_{s,a}(\pi)$ for all $a\in\cA$. 
Then,
\begin{equation}\label{eqn:v-Q-relation}
\V_s(\pi) = \sum_{a\in\cA}\pi_{s,a} Q_{s,a}(\pi) = \langle \pi_s,Q_s(\pi)\rangle,
\end{equation}
where $\langle\cdot,\cdot\rangle$ denotes the inner product of two vectors.

Policy gradients refer to the gradients of the value functions $\V_s(\pi)$ and $\V_\rho(\pi)$.
We can obtain their expressions as special cases of the policy gradient theorem \citep{Sutton2000PolicyGrad} which covers the general case with policy parametrization.
For easy reference, we give a simple, self-contained derivation in the Appendix (Section~\ref{sec:apdx:policy-grad}).
Specifically, we have
\begin{align}
\nablas \V_\rho(\pi)
&:=\frac{\partial \V_\rho(\pi)}{\partial \pi_s}
 = \frac{1}{1-\gamma} d_{\rho,s}(\pi) Q_{s}(\pi),
\label{eqn:policy-grad-s}
\end{align}
and $\nabla \V_\rho$ is the concatenation of $\nablas \V_\rho$ for all $s\in\cS$.
In other words, policy gradients are weighted $Q$-functions where the weights are block-diagonal and proportional to the discounted state-visitation probabilities.

A fundamental result for analyzing DMDP and related algorithms is the
performance difference lemma of \citet{Kakade2002icml}.
In this paper, we mostly rely on the following variant, which has appeared in \citet{LiuCaiYangWang2019NeurIPS} and \citet{Lan2021pmd}.
For completeness, here we provide an alternative proof.

\begin{lemma}[Performance difference lemma]\label{lem:pdl}
For any $\pi,\tilde{\pi}\in\Pi$, it holds that
\begin{equation}\label{eqn:pdl-s}
\V_s(\pi) - \V_s(\tilde\pi) = \frac{1}{1-\gamma} \E_{s'\sim d_s(\pi)} 
\left\langle Q_{s'}(\tilde\pi), \pi_{s'} - \tilde\pi_{s'}\right\rangle,
\qquad\forall\,s\in\cS.
\end{equation}
\end{lemma}
\begin{proof}
Using the relation~\eqref{eqn:v-Q-relation} on both $\pi$ and $\tilde\pi$ and the definition of $Q$-functions, we obtain
\begin{align*}
\V_s(\pi) - \V_s(\tilde\pi) 
&= \bigl\langle Q_s(\pi),\pi_s \bigr\rangle - \bigl\langle Q_s(\tilde\pi),\tilde\pi_s\bigr\rangle \\
&= \bigl\langle Q_s(\tilde\pi),\pi_s-\tilde\pi_s\bigr\rangle 
+\bigl\langle Q_s(\pi)-Q_s(\tilde\pi),\pi_s \bigr\rangle  \\
&= \bigl\langle Q_s(\tilde\pi),\pi_s-\tilde\pi_s\bigr\rangle 
 +\gamma \sum_{s\in\cS} \pi_{s,a}\sum_{s'\in\cS} P(s'|s,a)\bigl(\V_{s'}(\pi)-\V_{s'}(\Tilde{\pi})\bigr),
\qquad \forall\,s\in\cS.
\end{align*}
Define $u\in\R^\dS$ with components 
$u_s= \bigl\langle Q_s(\tilde\pi),\pi_s-\tilde\pi_s\bigr\rangle$. 
Then the above result leads to
\begin{align*}
\V(\pi) - \V(\tilde\pi) = u + \gamma P(\pi)\bigl(\V(\pi)-\V(\Tilde{\pi})\bigr),
\end{align*}
which further implies
$$\V(\pi) - \V(\tilde\pi) = \bigl(I-\gamma P(\pi)\bigr)^{-1} u.$$
Using the expression of $d_{s,s'}(\pi)$ in~\eqref{eqn:dsv-matrix}, we write the above equality component-wise as
\[
\V_s(\pi) - \V_s(\tilde\pi) 
= e_s^T\bigl(I-\gamma P(\pi)\bigr)^{-1} u
= \frac{1}{1-\gamma}\sum_{s'\in\cS} d_{s,s'}(\pi)\, u_{s'},
\]
which is the same as~\eqref{eqn:pdl-s}.
\end{proof}%

The weighted version of the performance difference lemma is,
\begin{equation}\label{eqn:pdl}
\V_\rho(\pi) - \V_\rho(\tilde\pi) = \frac{1}{1-\gamma} \E_{s'\sim d_\rho(\pi)} \left\langle Q_{s'}(\tilde\pi), \pi_{s'} - \tilde\pi_{s'}\right\rangle,
\qquad\forall\,\pi,\tilde\pi\in\Pi.
\end{equation}
Considering the expression of policy gradient in~\eqref{eqn:policy-grad-s}, the above characterization resembles that of a linear function
(precisely so if the expectation were taken with respect to $s'\sim d_\rho(\tilde\pi)$ instead of $s'\sim d_\rho(\pi)$).
The performance difference lemma is directly responsible for variational gradient domination \citep[Lemma~4]{Agarwal2021jmlr}, which is a convexity-like property, and also for descent with arbitrarily large step sizes (see Section~\ref{sec:pmd}), which is a concavity-like property.


\section{Projected Policy Gradient Method}
\label{sec:ppg}

In this section, we analyze the projected policy gradient method for solving the problem
\begin{equation}\label{eqn:min-v-rho}
\minimize_{\pi\in\Pi}~ \V_\rho(\pi),
\end{equation}
where $\V_\rho(\pi)$ is defined in~\eqref{eqn:v-rho-def} or equivalently~\eqref{eqn:v-rho-analytic}.
We assume that the policy gradients are computed with respect to an initial state distribution~$\mu$, which may be different from the performance evaluation distribution~$\rho$.

Starting from an initial policy $\pizero\in\Pi$, 
the projected policy gradient method generates a sequence $\pik$ for $k=1,2,\ldots$ as follows:
\begin{equation}\label{eqn:ppg}
\pikp = \Proj_{\Pi}\left(\pik - \eta_k\nabla \V_{\mu}(\pik)\right),
\end{equation}
where $\eta_k$ is the step size and $\Proj_\Pi(\cdot)$ denotes projection onto~$\Pi$ in Euclidean norm, i.e.,
$\Proj_\Pi(\pi) = \argmin_{\pi'\in\Pi}\|\pi'-\pi\|_2^2$.
Since $\Pi=\Delta(\cA)^\dS$ is a Cartesian product, the projections associated with different states can be done separately:
\begin{equation}\label{eqn:ppg-s}
\pikp_s = \Proj_{\Delta(\cA)}\left(\pik_s - \eta_k\nablas \V_{\mu}(\pik)\right), \qquad s\in\cS,
\end{equation}
where $\nablas \V_\mu$ is given by~\eqref{eqn:policy-grad-s} with $\rho$ replaced by~$\mu$.

\citet[Theorem~5]{Agarwal2021jmlr} show that with a constant step size $\eta_k = \frac{(1-\gamma)^3}{2\gamma\dA}$ for all $k\geq 0$, the projected policy gradient method converges at an $O\bigl(1/\sqrt{k}\bigr)$ rate. 
More precisely, 
\begin{equation}\label{eqn:ppg-slow-rate}
\min_{0<k\leq K}\left\{\V_\rho(\pik) - \V^\star_\rho\right\} \leq\epsilon
\quad\mbox{whenever}\quad
K > \frac{64\gamma\dS\dA}{\epsilon^2\,(1-\gamma)^6}\left\|\frac{d_\rho(\pi^\star)}{\mu}\right\|_\infty^2.
\end{equation}
The following two ingredients are key to their analysis.
\begin{itemize}
\item \textbf{Smoothness} \cite[Lemma~54]{Agarwal2021jmlr}: 
For any $\pi,\pi'\in\Pi$, it holds that
\begin{equation}\label{eqn:val-func-smooth}
\left\|\nabla \V_\rho(\pi)-\nabla \V_\rho(\pi')\right\|_2\leq\frac{2\gamma\dA}{(1-\gamma)^3}\left\|\pi-\pi'\right\|_2.
\end{equation}
\item \textbf{Variational gradient domination} \cite[Lemma~4]{Agarwal2021jmlr}:
For any $\pi\in\Pi$,
\begin{equation}\label{eqn:var-grad-dom}
\V_{\rho}(\pi) - \V_{\rho}^\star \leq \frac{1}{1-\gamma}\left\|\frac{d_{\rho}(\pi^\star)}{\mu}\right\|_\infty 
\max_{\pi'\in\Pi}\,\bigl\langle\nabla \V_{\mu}(\pi),\,\pi-\pi'\bigr\rangle.
\end{equation}
Here ``variational'' refers to the term $\max_{\pi'\in\Pi}\langle\nabla \V_\mu(\pi),\pi-\pi'\rangle$, which is different from $\|\nabla \V_\mu(\pi)\|$ as in gradient dominance conditions for unconstrained optimization.
\end{itemize}
Based on the same two results above, we show that the projected policy gradient method enjoys a faster $O(1/k)$ convergence rate. 
The following theorem holds for the case $\rho=\mu$.

\begin{theorem}\label{thm:ppg-rate}
Suppose $\rho=\mu$ 
and the step size $\eta_k = \frac{(1-\gamma)^3}{2\gamma\dA}$ for all $k\geq 0$. Then the projected policy gradient method~\eqref{eqn:ppg} generates a sequence of policies $\pik$ satisfying
\[
\V_\rho(\pik) - \V^\star_\rho 
~\leq~
\frac{128\dS\dA}{k\,(1-\gamma)^5} 
\left\|\frac{d_\rho(\pi^\star)}{\rho}\right\|^2_\infty.
\]
\end{theorem}

The general case with $\rho\neq\mu$ can be handled with an additional distribution mismatch coefficient. Concretely, 
\begin{align*}
\V_\rho(\pik)-\V_\rho^\star 
&= \sum_{s\in\cS} \rho_s\left(\V_s(\pik)-\V_s^\star\right) \\
&= \sum_{s\in\cS} \frac{\rho_s}{\mu_s}\mu_s\left(\V_s(\pik)-\V_s^\star\right) \\
&\leq \left\|\frac{\rho}{\mu}\right\|_\infty\left(\V_\mu(\pik)-\V_\mu^\star\right).
\end{align*}
Then applying Theorem~\ref{thm:ppg-rate} with~$\rho$ replaced by~$\mu$ yields
\[
\V_\rho(\pik)-\V_\rho^\star \leq \frac{128\dS\dA}{k(1-\gamma)^5}
\left\|\frac{d_\mu(\pistar)}{\mu}\right\|_\infty^2
\left\|\frac{\rho}{\mu}\right\|_\infty .
\]
Comparing with~\eqref{eqn:ppg-slow-rate}, in addition to improving the rate from $O(1/\sqrt{k})$ to $O(1/k)$, our bound also has better dependence on the discount factor: $1/(1-\gamma)^5$ as opposed to $1/(1-\gamma)^6$.
However, our bound uses a different distribution mismatch coefficient and has an additional factor of $\|\rho/\mu\|_\infty$.

We prove Theorem~\ref{thm:ppg-rate} as a special case of a more general result on \emph{gradient-mapping domination}, which we present next.

\subsection{An Interlude on Gradient-Mapping Domination}
\label{sec:grad-map-dom}

In this section, we consider the following composite optimization problem 
\begin{equation}\label{eqn:composite-opt}
\minimize_{x\in\R^n}~ F(x):= f(x) + \Psi(x),
\end{equation}
where $f$ is smooth and $\Psi$ is convex and lower semi-continuous.
More specifically, we assume that there exists a constant~$L>0$ such that 
\begin{equation}\label{eqn:f-smooth}
\|\nabla f(x)-\nabla f(y)\|_2 \leq L\|x-y\|_2, \qquad \forall\,x,y\in\dom\Psi.
\end{equation}
The MDP formulation in~\eqref{eqn:min-v-rho} is a special case of~\eqref{eqn:composite-opt} with the mappings $x\gets\pi$, $f\gets \V_\rho$ and $\Psi$ as the indicator function of $\Pi$, i.e., $\Psi(\pi)=0$ if $\pi\in\Pi$ and~$\infty$ otherwise.

For any convex function $\phi$, the $\prox$ operator is defined as
\[
\prox_\phi(x) = \argmin_y \left\{\phi(y)+\frac{1}{2}\|y-x\|_2^2\right\}.
\]
A generic algorithm for solving problem~\eqref{eqn:composite-opt} is the 
\emph{proximal gradient method}:
\begin{align}
\xkp &= \argmin_x \left\{
\left\langle\nabla f(\xk), x-\xk\right\rangle + \frac{1}{2\eta_k}\|x-\xk\|^2 + \Psi(x) \right\} \nonumber \\
&= \prox_{\eta_k\Psi }\left(\xk-\eta_k\nabla f(\xk)\right),
\label{eqn:prox-grad-method}
\end{align}
where $\eta_k$ is the step size.
If $\Psi$ is the indicator function of~$\Pi$, then $\prox_{\eta_k\Psi}$ becomes the projection operator $\Proj_{\Pi}$ for any $\eta_k>0$.
%
In this section, we focus on the proximal gradient method with the constant step size $\eta_k=1/L$.
To simplify presentation, we define
\begin{equation}\label{eqn:T-def}
T_L(x)
:=\prox_{\frac{1}{L}\Psi}\left(x-\frac{1}{L}\nabla f(x)\right),
\end{equation}
thus the proximal gradient method~\eqref{eqn:prox-grad-method} can be written simply as $\xkp=T_L(\xk)$.
The \emph{gradient mapping} associated with problem~\eqref{eqn:composite-opt} is defined as
\begin{equation}\label{eqn:G-def}
G_L(x) := L\bigl(x-T_L(x)\bigr).
\end{equation}
In the special case $\Psi\equiv 0$, we have $G_L(x)=\nabla f(x)$ for any $L>0$.
The norm of the gradient mapping, $\|G_L(\xk)\|$, can serve as a measure of closeness to a first-order stationary point.

Key to the convergence analysis of the proximal gradient method is the following descent property \citep[Theorem~1]{Nesterov2013composite}, 
\begin{equation}\label{eqn:grad-map-descent}
F(\xk)-F(\xkp) 
\geq \frac{1}{2L}\bigl\|G_L(\xk)\bigr\|_2^2.
\end{equation}
Summing up over $k=0,1,\ldots,K$, we obtain
\[
F(\xzero) - F(x^{(K+1)})\geq\frac{1}{2L}\sum_{k=0}^{K} \left\|G_L(\xk)\right\|_2^2,
\]
which, together with the fact $F(x^{(K+1)})\geq F^\star$, implies
\citep[Theorem~10.15]{Beck2017book}
\begin{equation}\label{eqn:grad-map-rate}
\min_{0\leq k\leq K}\left\|G_L(\xk)\right\|_2 \leq \sqrt{\frac{2L\left(F(\xzero)-F^\star\right)}{K+1}}.
\end{equation}
The $O(1/\sqrt{k})$ convergence rate stated in~\eqref{eqn:ppg-slow-rate} is obtained by combining~\eqref{eqn:grad-map-rate} with~\eqref{eqn:var-grad-dom}, which is the approach taken by \citet{Agarwal2021jmlr} and \citet{BhandariRusso2019global}.

We show that under similar conditions, the proximal gradient method actually enjoy a faster $O(1/k)$ rate of convergence.
To this end, the following notion of (weak) gradient-mapping domination is a proper extension of (weak) gradient domination.

\begin{definition}[\textbf{weak gradient-mapping domination}]
\label{def:weak-grad-dom}
Suppose $F:=f+\Psi$ where $f$ is $L$-smooth and $\Psi$ is proper, convex and closed. We say that~$F$ satisfies a \emph{weak gradient-mapping dominance} condition if there exists $\omega>0$ such that 
\begin{equation}\label{eqn:weak-grad-map-dom}
\left\|G_L(x)\right\|_2 \geq \sqrt{2\omega} \left(F\bigl(T_L(x)\bigr)-F^\star\right),
\qquad \forall\,x\in\dom\Psi,
\end{equation}
where $F^\star=\min_x F(x)$ and $T_L$ and $G_L$ are defined in~\eqref{eqn:T-def} and~\eqref{eqn:G-def} respectively.
\end{definition}
This \emph{weak} version of gradient-mapping domination
corresponds to the Kurdyka-{\L}ojasiewicz (K{\L}) condition
with K{\L} exponent~$1$, instead of the usual exponent~$1/2$ that leads to linear convergence \citep{Kurdyka1998, Karimi2016PL, LiPong2018}.
We discuss the stronger notion of gradient-mapping dominance in Appendix~\ref{sec:apdx:grad-map-dom}.

In the following theorem, we prove the $O(1/k)$ convergence rate for problems satisfying weak gradient-mapping domination. 

\begin{theorem}\label{thm:weak-grad-dom-rate}
Consider the problem of minimizing $F:=f+\Psi$ where $f$ is $L$-smooth 
and $\Psi$ is proper, convex and closed. 
Suppose~$F$ is weakly gradient-mapping dominant with parameter~$\omega$
and let $F^\star=\min_x F(x)$.
Then the proximal gradient method~\eqref{eqn:prox-grad-method} with a constant step size $\eta_k=1/L$ generates a sequence $\{\xk\}$ that satisfies, for all $k\geq 0$,
\begin{equation}\label{eqn:weak-grad-dom-rate}
F(\xk) - F^\star \leq \max\biggl\{\frac{4L}{\omega k},\,\biggl(\frac{\sqrt{2}}{2}\biggr)^k\bigl(F(\xzero)-F^\star\bigr)\biggr\} .
\end{equation}
\end{theorem}
\begin{proof}
Combining the descent property~\eqref{eqn:grad-map-descent} with the inequality~\eqref{eqn:weak-grad-map-dom} yields
\[
F(\xk) - F(\xkp) ~\geq~ \frac{1}{2L}\left\|G_L(\xk)\right\|_2^2
~\geq~ \frac{\omega}{L} \bigl(F(\xkp)-F^\star\bigr)^2.
\]
Let $\delta_k = F(\xk)-F^\star\geq 0$ 
Then we have
\[
\delta_k - \delta_{k+1} \geq\frac{\omega}{L}\delta_{k+1}^2 .
\]
We divide both sides of the above inequality by $\delta_k\delta_{k+1}$ to obtain
\[
\frac{1}{\delta_{k+1}}-\frac{1}{\delta_k}\geq \frac{\omega}{L}\frac{\delta_{k+1}}{\delta_k} .
\]
Then telescoping sum over iterations $0,1,\ldots,k-1$ yields
\[
\frac{1}{\delta_k}-\frac{1}{\delta_0} \geq \frac{\omega}{L}
\sum_{i=0}^{k-1}\frac{\delta_{i+1}}{\delta_i} .
\]
Notice that due to the descent property, we always have $\delta_{i+1}\leq\delta_i$ and thus $\delta_{i+1}/\delta_i\leq 1$.

For any two constant $r,c\in(0,1)$, let's define
$n(k,r)$ be the number of times that the ratio $\delta_{i+1}/\delta_i$ is at least~$r$ among the first~$k$ iterations.
If $n(k,r) \geq ck$, then $\delta_{i+1}/\delta_i\geq r$ at least $\lceil ck\rceil$ times, thus
\[
\frac{1}{\delta_{k}} - \frac{1}{\delta_0} \geq \frac{\omega}{L} r c k ,
\]
which implies
\[
\delta_k 
\leq \frac{1}{\frac{\omega}{L} r c k} 
= \frac{L}{\omega r c k} .
\]
Otherwise, we must have $n(k,r)< ck$, which means that $\delta_{i+1}/\delta_i<r$ at least $\lceil (1-c)k\rceil$ times. Noticing that $\delta_{i+1}/\delta_i\leq 1$ for all~$i$, we arrive at
\[
\delta_k\leq\delta_0 r^{(1-c)k} = \delta_0\left(r^{1-c}\right)^k .
\]
Combining the above two cases and using the fact that $r,c\in(0,1)$ can be chosen arbitrarily, we conclude that
\[
\delta_k \leq \min_{0<r,c<1} \max\left\{
\frac{L}{\omega r c k},
~\left(r^{1-c}\right)^k \right\} \delta_0 .
\]
Simply setting~$r=c=1/2$ gives the desired result~\eqref{eqn:weak-grad-dom-rate}. 
\end{proof}

\subsection{Proof of Theorem~\ref{thm:ppg-rate}}

In order to prove Theorem~\ref{thm:ppg-rate},
we only need to verify that the weak gradient-mapping domination 
holds for the weighted value function $\V_\rho$. This is the result of the next lemma.

\begin{lemma}\label{lem:var-to-dom}
Consider the problem of minimizing $\V_\rho$ over~$\Pi$ and suppose that $\V_\rho$ is $L$-smooth. 
We have
\begin{equation}\label{eqn:rho-mu-grad-dom}
\V_\rho\bigl(T_L(\pi)\bigr) - \V_\rho^\star \leq \frac{2\sqrt{2\dS}}{1-\gamma}\left\|\frac{d_\rho(\pi^\star)}{\mu}\right\|_\infty \left\|G_L(\pi)\right\|_2,
\end{equation}
where
\[
T_L(\pi) := \Proj_\Pi\left(\pi-\frac{1}{L}\nabla \V_\mu(\pi)\right),
\qquad
G_L(\pi) := L\bigl(\pi-T_L(\pi)\bigr).
\]
\end{lemma}
\begin{proof}
Applying a result of \citet[Theorem~1]{Nesterov2013composite} to our setting yields
\begin{align*}
\bigl\langle\nabla \V_\mu(T_L(\pi)),\, T_L(\pi)-\pi'\bigr\rangle
& \leq 2\|G_L(\pi)\|_2\cdot\|T_L(\pi)-\pi'\|_2 ,
\qquad\forall\,\pi,\pi'\in\Pi.
\end{align*}
Using the facts $T_L(\pi)\in\Pi$ and $\|\pi''-\pi'\|_2\leq\sqrt{2\dS}$ for any $\pi'',\pi'\in\Pi$, we obtain
\[
\max_{\pi'\in\Pi}\, \bigl\langle\nabla \V_\mu(T_L(\pi)), T_L(\pi)-\pi'\bigr\rangle \leq 2\sqrt{2\dS}\,\|G_L(\pi)\|_2.
\]
Combining the above bound with~\eqref{eqn:var-grad-dom} yields the desired result.
\end{proof}

An argument equivalent to Lemma~\ref{lem:var-to-dom} was used by \citet{Agarwal2021jmlr} which relies on a result of \citet{GhadimiLan2016}.
Combining~\eqref{eqn:rho-mu-grad-dom} with~\eqref{eqn:grad-map-rate} gives the result in~\eqref{eqn:ppg-slow-rate}.
On the other hand, we recognize that with $\rho=\mu$, inequality~\eqref{eqn:rho-mu-grad-dom} implies~\eqref{eqn:weak-grad-map-dom} with
\[
\omega = \frac{(1-\gamma)^2}{16\dS}\left\|\frac{d_\rho(\pistar)}{\rho}\right\|_\infty^{-2}, \qquad
L=\frac{2\gamma\dA}{(1-\gamma)^3} .
\]
Now we can apply Theorem~\ref{thm:weak-grad-dom-rate}.
Notice that in this case, the exponential decay part in~\eqref{eqn:weak-grad-dom-rate} is always smaller than the sublinear part, which leads to
$\V_\rho(\pik)-\V_\rho^\star\leq 4L/(\omega k)$, i.e.,
\[
\V_\rho(\pik)-\V_\rho^\star \leq \frac{128\dS\dA}{k(1-\gamma)^5}
\left\|\frac{d_\rho(\pistar)}{\rho}\right\|_\infty^2.
\]
This finishes the proof of Theorem~\ref{thm:ppg-rate}.

\citet[Theorem~5]{Junyu2020NeurIPS} have also established the $O(1/k)$ rate of the projected policy gradient method with direct parametrization. However, their proof appears to be quite different from ours, 
which leverages the dual linear programming parametrization 
\citep[Section~6.9]{Puterman1994book}.
In contrast, our approach is based on a novel notion of weak gradient-mapping domination and applies to general nonconvex composite optimization problems.

\section{Exact Policy Mirror Descent Methods}
\label{sec:pmd}

Mirror descent \citep{NemirovskiYudin1983book} is a general framework for the construction and analysis of optimization algorithms, which covers the projected gradient method as a special case. 
Here we adopt the form of mirror descent based on proximal minimization with respect to a Bregman divergence \citep{BeckTeboulle2003}.

Let $h:\Delta(\cA)\to\R$ be a strictly convex function and continuously differentiable on the relative interior of $\Delta(\cA)$, denoted as $\rint\Delta(\cA)$. 
The Bregman divergence generated by~$h$ is a distance-like function
defined as 
\[
D(p,p') := h(p) - h(p') - \langle \nabla h(p'), p-p'\rangle.
\]
Two most popular examples of Bregman divergence are:
\begin{itemize}
\item Squared Euclidean distance, generated by the squared $2$-norm: 
\[
h(p)=(1/2)\|p\|_2^2, \qquad D(p,p')=(1/2)\|p-p'\|_2^2.
\]
\item Kullback-Leibler (KL) divergence, generated by the negative entropy:
\[
h(p)=\sum_{a\in\cA}p_a\log p_a,
\qquad
D(p,p')=\sum_{a\in\cA}p_a\log\frac{p_a}{p'_a}.
\]
Notice that the gradient of negative entropy vanishes on the boundary of the simplex. Therefore we need to restrict the second argument $p'$ to lie within the relative interior of $\Delta(\cA)$.
We shall address such subtleties later in the convergence analysis.
\end{itemize}

Recall that the set of feasible policies is $\Pi=\Delta(\cA)^\dS$, which is a Cartesian product of $\dS$ copies of $\Delta(\cA)$.
For any $\rho\in\Delta(\cS)$, we define a weighted divergence function 
\[
D_{\rho}(\pi,\pi') := \E_{s\sim\rho}\bigl[D(\pi_s,\pi'_s)\bigr]
=\sum_{s\in\cS}\rho_s D(\pi_s,\pi'_s).
\]
This function satisfies the basic properties of a Bregman divergence; in particular, it is nonnegative and equals to~$0$ if and only if $\pi=\pi'$.

Following the derivations of \cite{Shani2020aaai}, we consider 
\emph{policy mirror descent} (PMD) methods with dynamically weighted divergences:
\[
\pikp = \argmin_{\pi\in\Pi} \left\{
\eta_k \bigl\langle \nabla \V_{\mu}(\pik),\,\pi\bigr\rangle
+ \frac{1}{1-\gamma}D_{d_{\mu}(\pik)}(\pi,\pik) \right\},
\]
where $\eta_k$ is the step size, $\mu\in\Delta(\cS)$ is an arbitrary state distribution and $d_\mu(\pik)$ is the discounted state-visitation distribution under the policy $\pik$.
Using the fact
\[
\bigl\langle\nabla \V_\mu(\pik),\,\pi\bigr\rangle=\sum_{s\in\cS}\bigl\langle\nabla_s \V_\mu(\pik),\,\pi_s\bigr\rangle,
\]
and plugging in the policy gradient formula~\eqref{eqn:policy-grad-s}, we obtain
\begin{align*}
\pikp &= \argmin_{\pi\in\Pi} \biggl\{
\frac{1}{1-\gamma}\sum_{s\in\cS}d_{\mu,s}(\pik) \Bigl(\eta_k\bigl\langle Q_s(\pik),\,\pi_s\bigr\rangle + D(\pi_s,\pik_s) \Bigr) \biggr\} \\
&= \argmin_{\pi\in\Pi} \biggl\{
\sum_{s\in\cS}\Bigl(\eta_k\bigl\langle Q_s(\pik),\,\pi_s\bigr\rangle + D(\pi_s,\pik_s) \Bigr) \biggr\},
\end{align*}
which can be written separately for each state as
\begin{equation}\label{eqn:pmd-update-s}
\pikp_s = \argmin_{p\in\Delta(\cA)} \Bigl\{
\eta_k\bigl\langle Q_s(\pik),\,p\bigr\rangle + D(p,\pik_s) \Bigr\}, \qquad \forall\,s\in\cS .
\end{equation}
Notice that the above update rule is independent of the choice of $\mu$.
This is the result of \emph{adaptive preconditioning} with a dynamically weighted divergence:
the weight for each state in $D_{d_\mu(\pik)}$
matches the coefficient of $Q_s(\pik)$ in the policy gradient $\nabla \V_\mu(\pik)$.

For the two prominent examples of Bregman divergence listed before, the corresponding PMD methods have closed-form update rules:
\begin{itemize}
\item \textbf{Projected $Q$-descent.}
If $D(\cdot,\cdot)$ is the squared Euclidean distance, then~\eqref{eqn:pmd-update-s} becomes
\begin{equation}\label{eqn:proj-Q-descent}
\pikp_s = \Proj_{\Delta(\cA)}\left(\pik_s-\eta_k Q_s(\pik)\right), 
\qquad \forall\,s\in\cS.
\end{equation}
Compared with the projected policy gradient method~\eqref{eqn:ppg-s}, we replaced the policy gradient $\nabla_s \V_\mu(\pik)$ by $Q_s(\pik)$ as the result of adaptive preconditioning.
\item \textbf{Exponentiated $Q$-descent.}
If $D(\cdot,\cdot)$ is the KL-divergence, then~\eqref{eqn:pmd-update-s} takes the form
\begin{equation}\label{eqn:exp-Q-descent}
\pikp_{s,a} = \pik_{s,a} \frac{\exp\left(-\eta_k Q_{s,a}(\pik)\right)}{\zk_s},  \qquad \forall\,a\in\cA, ~s\in\cS,
\end{equation}
where
\[
\zk_s = \sum_{a\in\cA}  \exp\left(-\eta_k Q_{s,a}(\pik)\right), \qquad\forall\, s\in\cS.
\]
This is exactly the Natural Policy Gradient (NPG) method 
\citep{Kakade2001NPG} expressed in the policy space \citep{Agarwal2021jmlr}.
\end{itemize}

In the rest of this section, we investigate the convergence rate of the PMD method~\eqref{eqn:pmd-update-s}. 
We show that with a constant step size, it has $O(1/k)$ convergence rate.
When the step size increases exponentially as $\eta_k=\eta_0/\gamma^k$, 
we have linear convergence and the convergnece rate depends on the distribution mismatch coefficient $\|d_\rho(\pistar)/\rho\|_\infty$.
In addition, we discuss situations of super-linear convergence
and connections to policy iteration.

Our results hold for PMD methods constructed with general Bregman divergences, matching or improving over the best known convergence rates.
In particular, the projected $Q$-descent method has the same rate of convergence as NPG. 
We show that the key ingredient for fast convergence of the PMD method is the adaptive preconditioning using weighted divergence functions.
The adopted local Bregman divergence, being KL-divergence or squared Euclidean distance, does not make much difference.

\subsection{Sublinear Convergence}
\label{sec:pmd-sublinear}

Our analysis is based on two key ingredients: the performance difference lemma (Lemma~\ref{lem:pdl}) and a three-point descent lemma on proximal optimization with Bregman divergences.

In order to cover both the squared Euclidean distance and KL-divergence without loss of rigor, we need some technical conditions. 
Specifically, we say a function~$h$ is of Legendre type \citep[Section~26]{Rockafellar1970book} if it is essentially smooth and strictly convex in the relative interior of $\dom h$, denoted as $\rint\dom h$. 
Essential smoothness means that~$h$ is differentiable and $\|\nabla h(x_k)\|\to\infty$ for every sequence $\{x_k\}$ converging to a boundary point of $\dom h$.
The following result is a slight variation of \citet[Lemma~3.2]{ChenTeboulle1993}, where we replaced the original assumption of $h$ being a \emph{Bregman function} with $h$ being of Legendre type.
The proof essentially follows the same arguments and thus is omitted here.

\begin{lemma}[Three-point descent lemma] 
\label{lem:3-point-descent}
Suppose that $\cC\subset\R^n$ is a closed convex set, $\phi:\cC\to\R$ is a proper, closed convex function, $D(\cdot,\cdot)$ is the Bregman divergence generated by a function~$h$ of Legendre type and $\rint\dom h\cap\cC\neq\emptyset$.
For any $x\in\rint\dom h$, let
\[
x^+=\argmin_{u\in\cC} \bigl\{\phi(u) + D(u,x)\bigr\}.
\]
Then $x^+\in\rint\dom h\cap\cC$ and for any $u\in\cC$,
\[
\phi(x^+)+D(x^+,x)\leq\phi(u)+D(u,x) - D(u,x^+).
\]
\end{lemma}

In the context of the PMD method~\eqref{eqn:pmd-update-s},
$\cC=\Delta(\cA)$ and $\phi$ is the linear function $\eta_k\langle Q_s(\pik),\,\cdot\,\rangle$.
There are some subtle differences between the two Bregman divergences we consider, as explained below.
\begin{itemize}
\item 
For the squared Euclidean distance, $h(\cdot)=(1/2)\|\cdot\|_2^2$ is of Legendre type with $\rint\dom h=\R^\dA$ and thus $\rint\dom h\cap\cC=\Delta(\cA)$. Therefore each iterate generated by the PMD method, specifically~\eqref{eqn:proj-Q-descent}, can be on the boundary of $\Delta(\cA)$.
\item
For the KL divergence, $h$ is the negative entropy function, which is also of Legendre type, but with 
$\rint\dom h\cap\cC=\rint\dom h=\rint\Delta(\cA)$.
Therefore, if we start with an initial point in $\rint\Delta(\cA)$, then every iterates will stay in $\rint\Delta(\cA)$.
\end{itemize}

We first use Lemma~\ref{lem:3-point-descent} to prove a descent property of PMD.
This result is elementary and has appeared in various forms before \citep[e.g.,][]{LiuCaiYangWang2019NeurIPS,Lan2021pmd}.
We present the proof for completeness as we will need to refer to some intermediate steps in it later.

\begin{lemma}[Descent property of PMD]
\label{lem:Q-v-descent}
Suppose the initial point $\pizero\in\rint\Pi$. 
Then the sequences generated by the PMD method~\eqref{eqn:pmd-update-s} satisfy
\begin{equation}\label{eqn:Q-descent}
\bigl\langle Q_s(\pik),\pikp_s-\pik_s\bigr\rangle  \leq 0, \qquad \forall\,s\in\cS,
\end{equation}
and for any $\rho\in\Delta(s)$, 
\begin{equation}\label{eqn:v-monotone}
\V_\rho(\pikp) \leq \V_\rho(\pik), \qquad \forall\, k\geq 0. 
\end{equation}
\end{lemma}

\begin{proof}
Applying Lemma~\ref{lem:3-point-descent} to the update rule~\eqref{eqn:pmd-update-s} with $\cC=\Delta(\cA)$ and $\phi(\cdot)=\eta_k\langle Q_s(\pi^k),\,\cdot\,\rangle$, we obtain that for any $p\in\Delta(\cA)$,
\[
\eta_k\llangle Q_s(\pik),\pikp_s\rrangle + D(\pikp_s,\pik_s) 
\leq \eta_k\llangle Q_s(\pik),p\rrangle + D(p,\pik_s)- D(p,\pikp_s).
\]
Rearranging terms and dividing both sides by~$\eta_k$, we get
\begin{equation}\label{eqn:3-point-result}
\langle Q_s(\pik),\pikp_s-p\rangle + \frac{1}{\eta_k} D(\pikp_s,\pik_s) 
\leq \frac{1}{\eta_k} D(p,\pik_s) - \frac{1}{\eta_k} D(p,\pikp_s).
\end{equation}
Letting $p=\pik_s$ in~\eqref{eqn:3-point-result} yields
\[
\langle Q_s(\pik),\pikp_s-\pik_s\rangle  
\leq - \frac{1}{\eta_k} D(\pikp_s,\pik_s) - \frac{1}{\eta_k} D(\pik_s,\pikp_s),
\]
which implies~\eqref{eqn:Q-descent} since the Bregman divergence $D(\cdot,\cdot)$ is always nonnegative.
By the performance difference lemma, specifically the weighted version~\eqref{eqn:pdl}, we have
\[
\V_\rho(\pikp) - \V_\rho(\pik) = \frac{1}{1-\gamma}\E_{s\sim d_\rho(\pikp)}
\bigl\langle Q_s(\pik),\pikp_s-\pik_s\bigr\rangle \leq 0,
\]
which is the same as~\eqref{eqn:v-monotone}.
\end{proof}

The next result is a generalization of the $O(1/k)$ convergence rate of the NPG method obtained by \citet[Theorem~16]{Agarwal2021jmlr}, where they focused on the setting of KL-divergence and their proof also relies on specific properties of the KL-divergence. 
Here we extend it to more general Bregman divergence. 
\citet[Theorem~2]{Lan2021pmd} derived a similar result using techniques that works for general Bregman divergence. However, he worked with the special objective function $\V_{\rho^\star}$ where $\rho^\star$ is the stationary distribution of the optimal policy $\pistar$.
As a result, the proof of~\citet[Theorem~2]{Lan2021pmd} avoids some subtle arguments required for the more general objective function $\V_\rho$ where $\rho\in\Delta(\cS)$ can be arbitrary.

In order to simplify presentation, we use the following notation throughout this paper:
\begin{equation}\label{eqn:Dk-def}
\Dk:=D_{d_\rho(\pistar)}(\pistar,\pik) = \sum_{s\in\cS}d_{\rho,s}(\pistar) D(\pistar_s,\pik_s),
\end{equation}
where $d_\rho(\pistar)$ is the state-visitation distribution under $\pistar$ with initial state distribution~$\rho$.
Although $\rho$ does not appear in the notation $\Dk$, we hope it is clear from the context.

\begin{theorem}
\label{thm:pmd-sublinear}
Consider the policy mirror descent method~\eqref{eqn:pmd-update-s} with 
$\pizero\in\rint\Pi$ and 
constant step size $\eta_k=\eta$ for all $k\geq 0$.
For any $\rho\in\Delta(\cS)$, we have for all $k\geq 0$,
\[
\V_\rho(\pik)-\V_\rho^\star 
\leq \frac{1}{k+1}\left(\frac{\Dzero
}{\eta(1-\gamma)}+\frac{1}{(1-\gamma)^2}\right).
\]
\end{theorem}

\begin{proof}
Consider the inequality~\eqref{eqn:3-point-result}, we let $p=\pistar_s$ and subtract and add $\pik_s$ within the inner product term, which leads to
\[
\langle Q_s(\pik),\pikp_s-\pik_s\rangle
+ \langle Q_s(\pik),\pik_s-\pistar_s\rangle
\leq \frac{1}{\eta_k} D(\pistar_s,\pik_s) - \frac{1}{\eta_k} D(\pistar_s,\pikp_s).
\]
Notice that we dropped the nonnegative term $(1/\eta_k) D(\pikp_s,\pik_s)$ on the left side of the inequality.
Taking expectation with respect to the distribution $d_\rho(\pistar)$
on both sides of the above inequality 
and using the notation in~\eqref{eqn:Dk-def}, we obtain
\begin{equation}\label{eqn:two-expectations}
\E_{s\sim d_\rho(\pistar)}\bigl[\langle Q_s(\pik),\pikp_s-\pik_s\rangle\bigr] + \E_{s\sim d_\rho(\pistar)}\big[\langle Q_s(\pik),\pik_s-\pistar_s\rangle\bigr] 
~\leq~ \frac{1}{\eta_k} \Dk - \frac{1}{\eta_k} \Dkp.
\end{equation}
For the first expectation in~\eqref{eqn:two-expectations}, we have
\begin{align}
\E_{s\sim d_\rho(\pistar)}\bigl[\langle Q_s(\pik),\pikp_s-\pik_s\rangle\bigr]
&=\sum_{s\in\cS} d_{\rho,s}(\pistar)\bigl\langle Q_s(\pik),\pikp_s-\pik_s\bigr\rangle \nonumber \\
&\geq \frac{1}{1-\gamma}\sum_{s\in\cS} d_{d_\rho(\pistar),\,s}(\pikp)\bigl\langle Q_s(\pik),\pikp_s-\pik_s\bigr\rangle \nonumber \\
&= \V_{d_\rho(\pistar)}(\pikp) - \V_{d_\rho(\pistar)}(\pik) ,
\label{eqn:sublinear-first-bd}
\end{align}
where the inequality holds because of~\eqref{eqn:Q-descent} 
and the fact, due to~\eqref{eqn:dsv-lower-bound}, that
\[
d_{d_\rho(\pistar),s}(\pikp)\geq(1-\gamma)d_{\rho,s}(\pistar),
\qquad\forall\,s\in\Delta(\cS).
\] 
The last equality in~\eqref{eqn:sublinear-first-bd} is due to the performance difference lemma. 
For the second expectation in~\eqref{eqn:two-expectations}, we again use the performance difference lemma to obtain
\begin{equation}\label{eqn:2nd-expectation}
\E_{s\sim d_\rho(\pistar)}\big[\langle Q_s(\pik),\pik_s-\pistar_s\rangle\bigr] ~=~ (1-\gamma)\bigl(\V_\rho(\pik)-\V_\rho(\pistar)\bigr).
\end{equation}
Substituting the two results above into~\eqref{eqn:two-expectations} leads to
\[
(1-\gamma)\left(\V_\rho(\pik)-\V_\rho(\pistar)\right)
\leq 
\frac{1}{\eta_k} \Dk - \frac{1}{\eta_k} \Dkp
+ \V_{d_\rho(\pistar)}(\pik) - \V_{d_\rho(\pistar)}(\pikp).
\]
Setting $\eta_k=\eta$ for all $k\geq 0$ and summing up over $k$:
\begin{align*}
(1-\gamma)\sum_{i=0}^k\left(\V_\rho(\pii)-\V_\rho(\pistar)\right)
& \leq \frac{1}{\eta}\Dzero - \frac{1}{\eta} \Dkp + \V_{d_\rho(\pistar)}(\pizero) - \V_{d_\rho(\pistar)}(\pikp)\\
& \leq \frac{1}{\eta}\Dzero + \V_{d_\rho(\pistar)}(\pizero).
\end{align*}
Since $\V_\rho(\pik)$ is monotone non-increasing in~$k$
(see Lemma~\ref{lem:Q-v-descent}), we conclude that
\[
\V_\rho(\pik)-\V_\rho^\star 
\leq \frac{1}{k+1} \sum_{i=0}^k\left(\V_\rho(\pii)-\V_\rho(\pistar)\right)
\leq \frac{1}{k+1}\left(\frac{\Dzero}{\eta(1-\gamma)} + \frac{\V_{d_\rho(\pistar)}(\pizero)}{1-\gamma}\right),
\]
Finally, bounding $\V_{d_\rho(\pistar)}(\pizero)$ by $1/(1-\gamma)$ as in~\eqref{eqn:value-bound} gives the desired result.
\end{proof}

As a result of Theorem~\ref{thm:pmd-sublinear}, 
whenever $\eta\geq(1-\gamma) D_{d_\rho(\pistar)}(\pistar,\pizero)$, we have 
\begin{equation}\label{eqn:sublinear-rate}
\V_\rho(\pik)-\V_\rho^\star \leq \frac{2}{(k+1)(1-\gamma)^2},
\qquad \forall\,k\geq 0.
\end{equation}
In other words, the number of iterations to reach $\V_\rho(\pik)-\V_\rho^\star \leq \epsilon$ is at most
$$\frac{2}{(1-\gamma)^2\epsilon},$$
which is independent of the problem dimensions $\dS$ and $\dA$.
More specifically,
\begin{itemize}
\item 
For the projected $Q$-descent method~\eqref{eqn:proj-Q-descent},
since $D(\pi_s,\pi'_s)=(1/2)\|\pi_s-\pi'_s\|^2\leq 1$ 
for any $\pi_s,\pi'_s\in\Delta(\cA)$, we have 
$D_\rho(\pi,\pi')=\sum_{s\in\cS}\rho_s D(\pi_s,\pi'_s)\leq 1$ for any $\rho\in\Delta(\cS)$.
Therefore in order for~\eqref{eqn:sublinear-rate} to hold, it suffices to have $\eta\geq(1-\gamma)$. 
\item
For the exponentiated $Q$-descent method~\eqref{eqn:exp-Q-descent}, if we choose the uniform initial policy, i.e., $\pizero_{s,a}=1/\dA$ for all $(s,a)\in\cS\times\cA$, then $D_\rho(\pistar,\pizero)\leq\log\dA$ for all $\rho\in\Delta(\cS)$.
Therefore in order for~\eqref{eqn:sublinear-rate} to hold, it suffices to have $\eta\geq(1-\gamma)\log\dA$.
\end{itemize}
The above analysis indicates that the projected $Q$-descent method may have a slight advantage over the exponentiated variant (NPG) in terms of having a wider range of~$\eta$ to enjoy the same dimensional independent convergence guarantee~\eqref{eqn:sublinear-rate}.

A more curious fact is that for both variants, the step size~$\eta$ does not have an upper bound and can be as large as possible. 
This is in contrast to the classical analysis of smooth optimization, where the step size is usually upper bounded by $2/L$ with $L$ being the Lipschitz constant of the gradient; see, e.g., the approach taken in Section~\ref{sec:grad-map-dom}.
Here the fact the step sizes can be arbitrarily large is due to the unique structure of DMDP.
Indeed, we show next that PMD has linear convergence if the step size grows exponentially.

\subsection{Linear Convergence}
\label{sec:pmd-linear}

Consider again the policy mirror descent algorithm~\eqref{eqn:pmd-update-s}. 
In order to simplify the presentation, we define two more notations: the optimality gap
\begin{equation}\label{eqn:Delta-k-def}
\delta_k := \V_\rho(\pik) - \V_\rho(\pistar), 
\end{equation}
which is always nonnegative, and the per-iteration distribution mismatch coefficient
\begin{equation}\label{eqn:theta-k-def}
\vartheta_{k}:=\left\|\frac{d_\rho(\pistar)}{d_\rho(\pik)}\right\|_\infty.
\end{equation}

The following result is the basis for establishing the linear convergence
and also for discussions on possible superlinear convergence.

\begin{proposition}\label{prop:master-recursion}
Consider the policy mirror descent method~\eqref{eqn:pmd-update-s} with $\pizero\in\rint\Pi$ and $\eta_k>0$ for all $k\geq 0$.
Then for any $\rho\in\Delta(\cS)$, we have for all $k\geq 0$,
\begin{equation}\label{eqn:master-recursion}
\vartheta_{k+1}\bigl(\delta_{k+1} - \delta_k\bigr) + \delta_k \leq 
\frac{1}{(1-\gamma)\eta_k} \Dk - \frac{1}{(1-\gamma)\eta_k} \Dkp,
\end{equation}
where $\delta_k$, $\vartheta_k$ and $\Dk$ are defined in~\eqref{eqn:Delta-k-def}, \eqref{eqn:theta-k-def} and~\eqref{eqn:Dk-def}, respectively.
\end{proposition}

\begin{proof}
We start with the inequality~\eqref{eqn:two-expectations} and bound
the first expectation as follows:
\begin{align*}
\E_{s\sim d_\rho(\pistar)}\bigl[\langle Q_s(\pik),\pikp_s-\pik_s\rangle\bigr]
&=\sum_{s\in\cS} d_{\rho,s}(\pistar)\bigl\langle Q_s(\pik),\pikp_s-\pik_s\bigr\rangle \\
&=\sum_{s\in\cS} \frac{d_{\rho,s}(\pistar)}{d_{\rho,s}(\pikp)} d_{\rho,s}(\pikp)\bigl\langle Q_s(\pik),\pikp_s-\pik_s\bigr\rangle \\
&\geq \left\|\frac{d_\rho(\pistar)}{d_\rho(\pikp)}\right\|_\infty
\sum_{s\in\cS} d_{\rho,s}(\pikp)\bigl\langle Q_s(\pik),\pikp_s-\pik_s\bigr\rangle \\
&= \left\|\frac{d_\rho(\pistar)}{d_\rho(\pikp)}\right\|_\infty
(1-\gamma) \bigl(\V_\rho(\pikp)-\V_\rho(\pik)\bigr) ,
\end{align*}
where the inequality holds because of~\eqref{eqn:Q-descent}, and the last equality is due to the performance difference lemma, specifically~\eqref{eqn:pdl}.
Substituting the above bound and~\eqref{eqn:2nd-expectation} into~\eqref{eqn:two-expectations} and dividing both sides by $1-\gamma$ yield the desired result.
\end{proof}

The next theorem is our main result on linear convergence.
The convergence rate depends on the performance evaluation distribution~$\rho$ through the following quantity:
\begin{equation}\label{eqn:theta-rho-def}
\vartheta_{\rho} :=
\frac{1}{1-\gamma}\left\|\frac{d_\rho(\pistar)}{\rho}\right\|_\infty,
\end{equation}
which is an upper bound on $\vartheta_k$ for all $k\geq 0$.

\begin{theorem}
\label{thm:pmd-linear}
Consider the policy mirror descent method~\eqref{eqn:pmd-update-s} with $\pizero\in\rint\Pi$.
Suppose 
the step sizes satisfy $\eta_0>0$ and
\begin{equation}\label{eqn:increasing-eta-k}
\eta_{k+1} \geq \frac{\vartheta_\rho}{\vartheta_\rho - 1}\eta_k, \qquad k=0,1,2,\ldots,
\end{equation}
then we have for each $k\geq 0$,
\begin{equation}\label{eqn:general-linear-rate}
\V_\rho(\pik) - \V_\rho^\star 
\leq \left(1-\frac{1}{\vartheta_\rho}\right)^k 
\left(\V_\rho(\pizero)-\V_\rho^\star+\frac{\Dzero}{\eta_0\gamma}\right).
\end{equation}
\end{theorem}

\begin{proof}
Using~\eqref{eqn:dsv-lower-bound}, 
specifically $d_{\rho,s}(\pik)\geq(1-\gamma)\rho_s$ for all $s\in\cS$,
we have $\vartheta_k\leq\vartheta_\rho$ for all $k\geq 0$.
In addition, by Lemma~\ref{lem:Q-v-descent}, we have $\delta_{k+1}-\delta_k\leq 0$ for all $k\geq 0$.
Therefore~\eqref{eqn:master-recursion} still holds if we replace~$\vartheta_{k+1}$ by its upper bound $\vartheta_\rho$, i.e.,
\[
\vartheta_\rho\bigl(\delta_{k+1} - \delta_k\bigr) + \delta_k \leq 
\frac{1}{(1-\gamma)\eta_k} \Dk -\frac{1}{(1-\gamma)\eta_k} \Dkp.
\]
Dividing both sides by $\vartheta_\rho$ and rearranging terms, we obtain
\[
\delta_{k+1}
+\frac{1}{(1-\gamma)\eta_k\vartheta_\rho} \Dkp
\leq \left(1-\frac{1}{\vartheta_\rho}\right)
\left(\delta_k + \frac{1}{(1-\gamma)\eta_k(\vartheta_\rho-1)} \Dk\right).
\]
If the step sizes satisfy~\eqref{eqn:increasing-eta-k}, i.e., 
$\eta_{k+1}(\vartheta_\rho-1)\geq\eta_k\vartheta_\rho$, 
then we have
\begin{align*}
\delta_{k+1}
+\frac{1}{(1\!-\!\gamma)\eta_{k+1}(\vartheta_\rho\!-\!1)} \Dkp
& \leq \left(1-\frac{1}{\vartheta_\rho}\right)
\left(\delta_k + \frac{1}{(1\!-\!\gamma)\eta_k(\vartheta_\rho\!-\!1)} \Dk \right).
\end{align*}
This forms a recursion and results in
\begin{equation}\label{eqn:pre-linear-bd}
\delta_k
+\frac{1}{(1\!-\!\gamma)\eta_k(\vartheta_\rho\!-\!1)} \Dk
 \leq \left(1-\frac{1}{\vartheta_\rho}\right)^{k}
\left(\delta_0 + \frac{1}{(1\!-\!\gamma)\eta_0(\vartheta_\rho\!-\!1)} \Dzero \right) .
\end{equation}
Finally, using the fact $\vartheta_\rho\geq 1/(1-\gamma)$, we derive
\begin{equation}\label{eqn:theta-rho-gamma}
(1-\gamma)(\vartheta_\rho-1) 
\geq (1-\gamma)\left(\frac{1}{1-\gamma}-1\right) = \gamma,
\end{equation}
Substituting the above bound into the right side of~\eqref{eqn:pre-linear-bd}, and considering the nonnegativity of $\Dk$ on the left side, we arrive at the desired bound~\eqref{eqn:general-linear-rate}.
\end{proof}

The exact value of $\vartheta_\rho$ is hard to estimate in practice, which hinders the use of the step size rule~\eqref{eqn:increasing-eta-k}. 
However, we can replace it with the more aggressive increasing rule 
\[
\eta_{k+1} = \eta_k/\gamma, \qquad \forall\,k\geq 0,
\]
which always implies~\eqref{eqn:increasing-eta-k}. 
To see this, we use $\vartheta_\rho \geq 1/(1-\gamma)$ to derive
\[
\frac{\vartheta_\rho}{\vartheta_\rho-1} 
\leq\frac{1/(1-\gamma)}{1/(1-\gamma)-1} 
= \frac{1}{\gamma}.
\]

According to Theorem~\ref{thm:pmd-linear}, 
in order to guarantee $\V_\rho(\pik)-\V_\rho^\star\leq\epsilon$,
the required number of iterations of the PMD method is
\[
\frac{1}{1-\gamma}\left\|\frac{d_\rho(\pistar)}{\rho}\right\|_\infty
\log\left(\left(\V_\rho(\pizero)-\V_\rho^\star + \frac{1}{\eta_0\gamma}\Dzero\right)\frac{1}{\epsilon}\right).
\]
Using the bound~\eqref{eqn:value-bound} and assuming
$\eta_0\geq \frac{1-\gamma}{\gamma}\Dzero$,
the iteration complexity becomes
\[
\frac{1}{1-\gamma}\left\|\frac{d_\rho(\pistar)}{\rho}\right\|_\infty
\log\frac{2}{(1-\gamma)\epsilon} .
\]
Next we discuss a special choice of the performance evaluation distribution~$\rho$.

\paragraph{Special case of $\rho=\rho^\star$.} 
Let $\rho^\star\in\Delta(\cS)$ be the stationary state distribution of the MDP under the optimal policy $\pistar$.
If the MDP starts with $s\sim\rho^\star$ and following $\pi^\star$, then the visit probability at every step is $\rho^\star$ and so is the discounted sum of them. Therefore we have $d_{\rho^\star}(\pistar)=\rho^\star$,
which implies 
\[
\left\|\frac{d_{\rho^\star}(\pistar)}{\rho^\star}\right\|_\infty=1 \qquad \mbox{and} \qquad \vartheta_{\rho^\star}=\frac{1}{(1-\gamma)}.
\]
In this case, with the step size rule $\eta_0\geq \frac{1-\gamma}{\gamma}\Dzero$ and
$\eta_{k+1}\geq\eta_k/\gamma$, we have
\[
\V_{\rho^\star}(\pik) - \V_{\rho^\star}(\pistar)
\leq \gamma^k \frac{2}{1-\gamma}
\]
and the iteration complexity for  
$\V_{\rho^\star}(\pik) - \V_{\rho^\star}(\pistar)\leq\epsilon$ is
dimension-independent:
\begin{equation}\label{eqn:linear-star-complexity}
\frac{1}{1-\gamma}\log\frac{2}{(1-\gamma)\,\epsilon}.
\end{equation}
However, unless the MDP is ergodic, the support of $\rho^\star$ may not cover the full state space~$\cS$.

Several recent work studied policy mirror descent method for entropy-regularized MDP and obtained similar linear convergence rates
\citep{Cen2020NPGentropy,Lan2021pmd,Zhan2021Regularized}.
With entropy regularization, the resulting MDP is always ergodic and the support of any stationary distribution covers the full state space~$\cS$, i.e., $\rho^\star> 0$.
\citet{Lan2021pmd} only considers $\rho^\star$ as the performance evaluation distribution; \citet{Cen2020NPGentropy} and \citet{Zhan2021Regularized}
rely on the contraction properties of a generalized Bellman operator and obtain guarantees of the form $\|Q(\pik)-Q(\pistar)\|_\infty\leq\epsilon$ where $Q$ is the ``soft'' $Q$-function with regularization. 
Our analysis closely resembles that of \citet{Lan2021pmd}, with the following differences:
\begin{itemize}
\item We consider the standard DMDP and show that linear convergence can be obtained without entropy regularization.
Since the support of $\rho^\star$ may not cover the entire state space, we give a general analysis for any $\rho\in\Delta(\cS)$ and characterize the convergence rate in terms of the distribution mismatch coefficient 
$\|d_\rho(\pistar)/\rho\|_\infty$.
\item For DMDP without regularization, \citet{Lan2021pmd} also obtains a slower linear convergence rate ($\gamma^{k/2}$ instead of $\gamma^k$), through an approximate policy mirror descent (APMD) method. This method employs exponentially diminishing regularization and exponentially increasing step sizes, and the analysis is considerably more technical. 
\end{itemize}

\subsection{Superlinear Convergence}
\label{sec:pmd-superlinear}

Under additional conditions, the PMD method~\eqref{eqn:pmd-update-s} may exhibit superlinear convergence.
We revisit Proposition~\ref{prop:master-recursion} and start by rewriting the inequality~\eqref{eqn:master-recursion} as
\[
\delta_{k+1} 
+ \frac{\Dkp}{(1-\gamma)\vartheta_{k+1}\eta_k} 
\leq \left(1-\frac{1}{\vartheta_{k+1}} \right) \left(\delta_k
+\frac{\Dk}{(1-\gamma)(\vartheta_{k+1}-1)\eta_k} \right).
\]
If the step sizes satisfy $\eta_k\geq\frac{\vartheta_k}{\vartheta_{k+1}-1}\eta_{k-1}$ starting with some $\eta_{-1}>0$, then we have
\begin{align}
\delta_{k+1} + \frac{\Dkp}{(1-\gamma)\vartheta_{k+1}\eta_k} 
&\leq \left(1-\frac{1}{\vartheta_{k+1}} \right) \left(\delta_k
+\frac{\Dk}{(1-\gamma)\vartheta_k\eta_{k-1}}\right) \nonumber \\
&\leq \prod_{i=0}^k \left(1-\frac{1}{\vartheta_{i+1}} \right) \left(\delta_0
+\frac{\Dzero}{(1-\gamma)\vartheta_0\eta_{-1}} \right) .
\label{eqn:superlinear-recursion}
\end{align}
Therefore, we have superlinear convergence of $\delta_k$ if $\vartheta_k\to 1$.

Recall the definition of $\vartheta_k$ in~\eqref{eqn:theta-k-def}, 
we have $\vartheta_k\to 1$ if and only if $d_\rho(\pik)\to d_\rho(\pistar)$.
Apparently, a sufficient condition is $\pik\to\pistar$.
However, this is hard to establish without additional assumptions, e.g., by assuming that the optimal policy $\pistar$ is unique.
Alternatively, since
$\Dk=D_{d_\rho(\pi^\star)}(\pi^\star,\pi^k)\to 0$ implies
$\pi^k\to\pi^\star$, 
a reasonable attempt is to show the convergence of $\Dk$
by further leveraging~\eqref{eqn:superlinear-recursion}.
In particular, we can show
$$\frac{\Dk}{(1-\gamma)\vartheta_k\eta_{k-1}}\to 0$$
at the same speed as $\delta_k\to 0$, which is at least linear with an uniform upper bound on $\vartheta_k$ as we have done in Section~\ref{sec:pmd-linear}.
However, the step-size condition $\eta_k\geq\frac{\vartheta_k}{\vartheta_{k+1}-1}\eta_{k-1}$ implies that the factor $1/(\vartheta_k\eta_{k-1})$ itself converges at the same rate, thus we can not guarantee $\Dk\to 0$.

Nevertheless, we list here two sufficient conditions for superlinear convergence that are weaker than directly assuming $\pik\to\pistar$.
Both conditions have been used to establish superlinear convergence of the classical Policy Iteration algorithm
\citep[Corollary~6.4.10 and Theorem~6.4.8, respectively]{Puterman1994book}.
\begin{itemize} 
\item Convergence of the transition probability matrix $P(\pi^k)$. Specifically,
\[
\lim_{k\to\infty} \bigl\| P(\pi^k) - P(\pi^\star)\bigr\| = 0,
\]
where $\|\cdot\|$ is any matrix norm.
Under this condition, we have $d_\rho(\pik)\to d_\rho(\pistar)$ and thus $\vartheta_k\to 1$ because
$d_\rho(\pik) = \bigl(I - \gamma P(\pik)\bigr)^{-T}\!\rho$ is a continuous function.
\item There exists a finite constant $C>0$ such that for all $k=1,2,\ldots$
\[
\bigl\| P(\pi^k) - P(\pi^\star)\bigr\| \leq C \bigl( \V_\rho(\pi^k) - \V_\rho^\star \bigr) .
\]
This condition is stronger than the previous one because we already established linear convergence of $\V_\rho(\pi^k) - \V_\rho^\star $.
As a result, it leads to local quadratic convergence.
\end{itemize}
\citet{Khodadadian2021} showed that under a variant of the second condition above, the NPG method converges superlinerly. 
With entropy regularization, the optimal policy $\pistar$ is unique and \citet{Cen2020NPGentropy} established local quadratic convergence of the regularized PMD method.

\subsection{Connection with Policy Iteration}
\label{sec:pmd-connection}

Our analysis of the PMD method does not impose any upper bound on the step sizes: they can be either arbitrarily large constant (Section~\ref{sec:pmd-sublinear}) or gemmetrically increasing (Section~\ref{sec:pmd-linear}).
If we allow $\eta_k\to\infty$ for all iterations, the limit of the PMD method~\eqref{eqn:pmd-update-s} becomes
\[
\pikp_s = \argmin_{p\in\Delta(\cA)} 
~\bigl\langle Q_s(\pik),\,p\bigr\rangle, 
\qquad \forall\,s\in\cS ,
\]
which is precisely the classical Policy Iteration method
\citep[e.g.,][]{Puterman1994book,Bertsekas2012bookDPOC}.
In fact, our analysis still holds in the limiting case and the result corresponding to Theorem~\ref{thm:pmd-linear} is
\[
\delta_{k+1} \leq \left(1-\frac{1}{\vartheta_\rho}\right)^k \delta_0,
\]
where $\delta_k=\V_\rho(\pik)-\V_\rho^\star$.
Recall the definition of $\vartheta_\rho$ in~\eqref{eqn:theta-rho-def}. 
If $\rho=\rho^\star$, then we have $\vartheta_{\rho^\star}=1/(1-\gamma)$ and 
\[
\delta_{k+1} \leq \gamma^k \delta_0,
\]
which has the same convergence rate as Policy Iteration
\citep[e.g.,][]{Puterman1994book,Ye2011MDP}.
In general, we have the trivial bound 
\[
\vartheta_\rho = \frac{1}{1-\gamma}\left\|\frac{d_\rho(\pistar)}{\rho}\right\|_\infty \leq\frac{1}{(1-\gamma)\min_{s\in\cS}\rho_s},
\]
which leads to
\[
\delta_{k+1} \leq \left(1-(1-\gamma)\min_{s\in\cS}\rho_s\right)^k \delta_0.
\]
This convergence rate is the same as that established for several variants of policy gradient methods by \citet[Theorem~1]{BhandariRusso2021linear}, which requires exact line search. 
\citet{Khodadadian2021} show that the NPG method with an adaptive step size rule can also achieve linear convergence. 
In contrast, our results in Section~\ref{sec:pmd-linear} show that the simple, non-adaptive step size schedule of $\eta_k=\eta_0/\gamma^k$ is suffice to obtain linear convergence of a general class of policy mirror descent methods.


\section{Inexact Policy Mirror Descent Methods}
\label{sec:inexact-pmd}

For DMDP problems with large state and action spaces, computing the exact policy gradients or $Q$-functions are very costly and infeasible in practice. In this section, we consider the following inexact PMD method
\begin{equation}\label{eqn:inexact-pmd-s}
\pikp_s = \argmin_{p\in\Delta(\cA)}\left\{\eta_k\llangle \hQ_s(\pik),p\rrangle + D(p,\pik_s)\right\}, 
\qquad \forall\,s\in\cS.
\end{equation}
where $\hQ_s(\pik)$ is an inexact evaluation of $Q_s(\pik)$.
We first study the convergence properties of~\eqref{eqn:inexact-pmd-s} under the following assumption on the evaluation error.

\begin{assumption}\label{asmp:inexact-Q}
The inexact $Q$-function evaluations $\hQ(\pik)$ satisfy
\begin{equation}\label{eqn:inexact-Q-eval}
\bigl\| \hQ(\pik)-Q(\pik)\bigr\|_\infty \leq\tau,\qquad\forall\,k\geq 0.
\end{equation}
\end{assumption}

The following result is the counterpart of Lemma~\ref{lem:Q-v-descent} for the inexact PMD method.

\begin{lemma}\label{lem:inexact-Q-v-descent}
Consider the inexact PMD method~\eqref{eqn:inexact-pmd-s} with $\pizero\in\rint\Pi$ and suppose that Assumption~\ref{asmp:inexact-Q} holds. 
Then we have for all $k\geq 0$,
\begin{equation}\label{eqn:inexact-Q-descent}
\bigl\langle \hQ_s(\pik),\pikp_s-\pik_s\bigr\rangle \leq 0, \qquad \forall\,s\in\cS,
\end{equation}
and for any $\rho\in\Delta(\cS)$,
\begin{equation}\label{eqn:inexact-v-descent}
\V_\rho(\pikp) - \V_\rho(\pik) \leq \frac{2}{1-\gamma}\tau.
\end{equation}
\end{lemma}

\begin{proof}
The proof of~\eqref{eqn:inexact-Q-descent} follows the same arguments as in Lemma~\ref{lem:Q-v-descent}.
However, due to the inexact $Q$-function evaluations, 
the objectives $\V_\rho(\pik)$ are no longer monotone decreasing. 
We use the performance difference lemma to deduct:
\begin{align*}
\V_\rho(\pikp) - \V_\rho(\pik) 
&= \frac{1}{1-\gamma}\sum_{s\in\cS}d_{\rho,s}(\pikp)\llangle Q_s(\pik),\pikp_s-\pik_s\rrangle \\
&= \frac{1}{1-\gamma}\sum_{s\in\cS}d_{\rho,s}(\pikp)\llangle \hQ_s(\pik),\pikp_s-\pik_s\rrangle \\
&\quad + \frac{1}{1-\gamma}\sum_{s\in\cS}d_{\rho,s}(\pikp)\llangle Q_s(\pik)-\hQ_s(\pik),\pikp_s-\pik_s\rrangle .
\end{align*}
Notice that the first term on the right-hand side is non-positive due to~\eqref{eqn:inexact-Q-descent}. 
For the second term, we use H\"older's inequality to obtain, for all $s\in\cS$,
\begin{align}
\llangle Q_s(\pik)-\hQ_s(\pik),\pikp_s-\pik_s\rrangle 
& \leq \bigl\|Q_s(\pik)-\hQ_s(\pik)\bigr\|_\infty \bigl\|\pikp_s-\pik_s\bigr\|_1 \nonumber \\
&\leq 2 \bigl\|\hQ_s(\pik)-Q_s(\pik)\bigr\|_\infty \nonumber\\
&\leq 2 \tau,
\label{eqn:Q-holder}
\end{align}
where the second inequality is due to 
$\bigl\|\pikp_s-\pik_s\bigr\|_1\leq\bigl\|\pikp_s\bigr\|_1+\bigl\|\pik_s\bigr\|_1\leq 2$,
and the last inequality is due to Assumption~\ref{asmp:inexact-Q}. 
Combining~\eqref{eqn:Q-holder} with the previous inequality yields~\eqref{eqn:inexact-v-descent}.
\end{proof}

We will need the following simple fact, whose proof is straightforward and thus omitted.
\begin{lemma}\label{lem:simple-series}
Suppose $0<\alpha<1$, $b>0$, and a nonnegative sequence $\{a_k\}$ satisfies
\[
a_{k+1} \leq \alpha a_k + b, \qquad\forall\,k\geq 0.
\]
Then for all $k\geq 0$,
\[
a_k \leq \alpha^k a_0 + \frac{b}{1-\alpha}.
\]
\end{lemma}

The following theorem characterizes the convergence of the inexact PMD method under Assumption~\ref{asmp:inexact-Q}.
We keep using the notations $\Dk$ and $\vartheta_\rho$ defined in~\eqref{eqn:Dk-def} and~\eqref{eqn:theta-rho-def}, respectively.

\begin{theorem}\label{thm:inexact-linear-rate}
Consider the inexact PMD method~\eqref{eqn:inexact-pmd-s} with $\pizero\in\rint\Pi$ and suppose that Assumption~\ref{asmp:inexact-Q} holds. 
If the step sizes satisfy $\eta_0\geq\frac{1-\gamma}{\gamma}\Dzero$ and $\eta_{k+1}\geq\eta_k/\gamma$, then we have for all $k\geq 0$,
\begin{equation}\label{eqn:inexact-linear-rate}
\V_\rho(\pik) - \V_\rho^\star 
\leq \biggl(1-\frac{1}{\vartheta_\rho}\biggr)^k \frac{2}{1-\gamma}
+ \frac{4\vartheta_\rho}{1-\gamma}\tau.
\end{equation}
\end{theorem}

\begin{proof}
Applying Lemma~\ref{lem:3-point-descent} to the update in~\eqref{eqn:inexact-pmd-s} and following the same arguments in the proof of Theorem~\ref{thm:pmd-sublinear}, we arrive at the following counterpart of~\eqref{eqn:two-expectations}:
\begin{equation}\label{eqn:inexact-two-expect}
\E_{s\sim d_\rho(\pistar)}\bigl[\langle \hQ_s(\pik),\pikp_s-\pik_s\rangle\bigr] + \E_{s\sim d_\rho(\pistar)}\big[\langle \hQ_s(\pik),\pik_s-\pistar_s\rangle\bigr] 
~\leq~ \frac{1}{\eta_k} \Dk - \frac{1}{\eta_k} \Dkp.
\end{equation}
For the first expectation in~\eqref{eqn:inexact-two-expect}, we follow the proof of Proposition~\ref{prop:master-recursion} to obtain
\begin{align*}
\E_{s\sim d_\rho^\star}\bigl[\langle \hQ_s(\pik),\pikp_s-\pik_s\rangle\bigr]
&\geq \left\|\frac{d_\rho(\pistar)}{d_\rho(\pikp)}\right\|_\infty
\sum_{s\in\cS} d_{\rho,s}(\pikp)\bigl\langle \hQ_s(\pik),\pikp_s-\pik_s\bigr\rangle \\
&=\vartheta_{k+1}\sum_{s\in\cS} d_{\rho,s}(\pikp)\bigl\langle Q_s(\pik),\pikp_s-\pik_s\bigr\rangle  \\
&\quad +\vartheta_{k+1}\sum_{s\in\cS} d_{\rho,s}(\pikp)\bigl\langle \hQ_s(\pik)-Q_s(\pik),\pikp_s-\pik_s\bigr\rangle \\
&\geq \vartheta_{k+1} (1-\gamma) \bigl(\V_\rho(\pikp)-\V_\rho(\pik)\bigr) 
-2\vartheta_{k+1}\tau,
\end{align*}
where the last inequality is due to the performance difference lemma and~\eqref{eqn:Q-holder}.
For the second expectation in~\eqref{eqn:inexact-two-expect}, we again use the performance difference lemma and H\"older's inequality to obtain
\begin{align*}
&\quad \E_{s\sim d_\rho(\pistar)}\big[\langle \hQ_s(\pik),\pik_s-\pistar_s\rangle\bigr]  \\
&=
\E_{s\sim d_\rho(\pistar)}\big[\langle Q_s(\pik),\pik_s-\pistar_s\rangle\bigr] +  
\E_{s\sim d_\rho(\pistar)}\big[\langle \hQ_s(\pik)-Q_s(\pik),\pik_s-\pistar_s\rangle\bigr] \\
&\geq (1-\gamma)\bigl(\V_\rho(\pik)-\V_\rho(\pistar)\bigr) 
-\E_{s\sim d_\rho(\pistar)}\big[\bigl\|\hQ_s(\pik)-Q_s(\pik)\bigr\|_\infty \bigl\|\pik_s-\pistar_s\bigr\|_1\bigr] \\
&\geq (1-\gamma)\bigl(\V_\rho(\pik)-\V_\rho(\pistar)\bigr) - 2\tau.
\end{align*}
Substituting the last two bounds into~\eqref{eqn:inexact-two-expect} and dividing both sides by $1-\gamma$, we get 
\[
\vartheta_{k+1}\left(\delta_{k+1}-\delta_k-\frac{2\tau}{1-\gamma}\right) + \delta_k
\leq \frac{1}{(1-\gamma)\eta_k}\Dk - \frac{1}{(1-\gamma)\eta_k}\Dkp
+\frac{2\tau}{1-\gamma},
\]
where $\delta_k := \V_\rho(\pikp)-\V_\rho^\star$.
Since $\delta_{k+1}-\delta_k-\frac{2\tau}{1-\gamma}\leq 0$ (Lemma~\ref{lem:inexact-Q-v-descent}) and $\vartheta_{k+1}\leq\vartheta_\rho$, the above inequality still holds with $\vartheta_{k+1}$ replaced by $\vartheta_\rho$, which leads to
\[
\vartheta_\rho\left(\delta_{k+1}-\delta_k\right) + \delta_k
\leq \frac{1}{(1-\gamma)\eta_k}\Dk - \frac{1}{(1-\gamma)\eta_k}\Dkp
+\frac{2(1+\vartheta_\rho)\tau}{1-\gamma}.
\]
Dividing both sides by $\vartheta_\rho$ and rearranging terms, we get
\[
\delta_{k+1} + \frac{1}{(1-\gamma)\eta_k\vartheta_\rho}\Dkp 
\leq \left(1-\frac{1}{\vartheta_\rho}\right)
\left(\delta_k + \frac{1}{(1-\gamma)\eta_k(\vartheta_\rho-1)}\Dk\right)
+\left(1+\frac{1}{\vartheta_\rho}\right)\frac{2\tau}{1-\gamma}.
\]
If the step sizes satisfy $\eta_{k+1}(\vartheta_\rho-1)\geq\eta_k\vartheta_\rho$, which is implied by $\eta_{k+1}\geq\eta_k/\gamma$, then
\[
\delta_{k+1} + \frac{1}{(1-\gamma)\eta_{k+1}(\vartheta_\rho-1)}\Dkp 
\leq \left(1-\frac{1}{\vartheta_\rho}\right)
\left(\delta_k + \frac{1}{(1-\gamma)\eta_k(\vartheta_\rho-1)}\Dk\right)
+\frac{4\tau}{1-\gamma},
\]
where we also used $1+1/\vartheta_\rho<2$ because $\vartheta_\rho>1$.
Next we invoke Lemma~\ref{lem:simple-series} with
\[
a_k = \delta_k + \frac{1}{(1-\gamma)\eta_k(\vartheta_\rho-1)}\Dk,
\qquad \alpha = 1-\frac{1}{\vartheta_\rho}
\qquad \mbox{and} \qquad b=\frac{4\tau}{1-\gamma},
\]
which leads to
\[
\delta_k \leq \left(1-\frac{1}{\vartheta_\rho}\right)^k \left(\delta_0+\frac{1}{(1-\gamma)\eta_0(\vartheta_\rho-1)}\Dzero\right) + \frac{4\vartheta_\rho}{1-\gamma}\tau.
\]
Finally applying~\eqref{eqn:theta-rho-gamma} and $\eta_0\geq\frac{1-\gamma}{\gamma}\Dzero$ gives the desired result~\eqref{eqn:inexact-linear-rate}.
\end{proof}

As a result of Theorem~\ref{thm:inexact-linear-rate}, we have the following asymptotic error bound:
\[
\lim_{k\to\infty} \V_\rho(\pik) - \V_\rho^\star 
~\leq~ \frac{4\vartheta_\rho}{1-\gamma}\tau
~=~ \frac{4\tau}{(1-\gamma)^2}\left\|\frac{d_\rho(\pistar)}{\rho}\right\|_\infty,
\]
which agrees with that of conservative policy iteration (CPI) of \citet[Theorem~6.2]{Kakade2002icml}. 
It is also similar to the asymptotic error bound of many approximate dynamical programming algorithms \citep[e.g.,][]{Bertsekas2012bookDPOC}, with the additional factor of distribution mismatch coefficient.

\subsection{Sample Complexity under a Generative Model}
\label{sec:high-prob}

One way to ensure Assumption~\ref{asmp:inexact-Q} hold with high probability is through multiple independent simulations (rollouts) of the MDP under a fixed policy.
In this section, we analyze the sample complexity of this approach.

Suppose that for a given policy $\pik$ and any state-action pair $(s,a)\in\cS\times\cA$, we can generate a set of $M_k$ independent, truncated trajectories of horizon $H$, i.e.,
\[
\cT_{s,a}^{(k,i)} = \left\{(s_0^{(i)},a_0^{(i)}),(s_1^{(i)},a_1^{(i)}),\ldots,(s_{H-1}^{(i)},a_{H-1}^{(i)})~\Big|~ s_0^{(i)}=s,~a_0^{(i)}=a
\right\},
\quad i=1,\ldots,M_k.
\]
We construct $\hQ_{s,a}(\pik)$ with the trajectories $\cT_{s,a}^{(k,i)}$,
$i=1,\ldots,M_k$, as follows:
\begin{equation}\label{eqn:Q-hat}
\hQ_{s,a}(\pik) := \frac{1}{M_k}\sum_{i=1}^{M_k}\hQ_{s,a}^{(i)}(\pik), \quad\mbox{where}\quad
\hQ_{s,a}^{(i)}(\pik):=\sum_{t=0}^{H-1}\gamma^t R\bigl(s_t^{(i)},\,a_t^{(i)}\bigr) .
\end{equation}

The following lemma gives a high-probability bound on the error $\|\hQ(\pik)-Q(\pik)\|_\infty$.

\begin{lemma}\label{lem:Q-high-prob-bd}
Consider the $Q$-estimator given in~\eqref{eqn:Q-hat}.
For any $\delta\in(0,1)$, if $M_k$ satisfies
\[
M_k \geq \frac{\gamma^{-2 H}}{2}\,\log\left(\frac{2\dS\dA}{\delta}\right),
\]
then we have with probability at least $1-\delta$,
\begin{equation}\label{eqn:Q-high-prob-bd}
\left\|\hQ(\pik)-Q(\pik)\right\|_\infty 
\leq \frac{2\gamma^{H}}{1-\gamma}.
\end{equation}
\end{lemma}

\begin{proof}
We first define the expectation of the $Q$-estimator in~\eqref{eqn:Q-hat}:
\[
\bQ_{s,a}(\pik) :=  \E\left[\hQ_{s,a}(\pik)\right] 
= \E\left[\hQ_{s,a}^{(i)}(\pik)\right] 
= \E\left[\sum_{t=0}^{H-1}\gamma^t R\bigl(s_t^{(i)},\,a_t^{(i)}\bigr)\right],
\]
which holds for any $i=1,\ldots,M_k$.
Recall the definition of $Q_{s,a}$ in~\eqref{eqn:Q-def}.
Since $R(s,a)\geq 0$, we always have
$Q_{s,a}(\pik) - \bQ_{s,a}(\pik)\geq 0$.
On the other hand,
\begin{align*}
Q_{s,a}(\pik) - \bQ_{s,a}(\pik) 
&=\E\left[\sum_{t=H}^{\infty}\gamma^t R\bigl(s_t^{(i)},\,a_t^{(i)}\bigr)\right]
\leq \E\left[\sum_{t=H}^{\infty}\gamma^t \right]
= \frac{\gamma^{H}}{1-\gamma},
\end{align*}
which holds for all $(s,a)\in\cS\times\cA$.
Therefore, 
\begin{equation}\label{eqn:hp-bias-bd}
\left\|\bQ(\pik)-Q(\pik)\right\|_\infty 
\leq \frac{\gamma^{H}}{1-\gamma}.
\end{equation}

Next that we can decompose the estimation error into two parts:
\begin{equation}\label{eqn:hp-error-decomp}
\left\|\hQ(\pik)-Q(\pik)\right\|_\infty 
\leq \left\|\hQ(\pik)-\bQ(\pik)\right\|_\infty 
+ \left\|\bQ(\pik)-Q(\pik)\right\|_\infty .
\end{equation}
The last term is bounded by~\eqref{eqn:hp-bias-bd}, so we need to bound  $\bigl\|\hQ(\pik)-\bQ(\pik)\bigr\|_\infty$.
To this end, we notice that the random variables $\hQ_{s,a}^{(i)}(\pik)$ are bounded in the interval $[0,1/(1-\gamma)]$.
Therefore Hoeffding's inequality \citep{Hoeffding1963} implies that 
for any $\sigma_k>0$,
\begin{align*}
\Prob\left(\left|\hQ_{s,a}(\pik)-\bQ_{s,a}(\pik)\right|\geq\sigma_k\right) 
& \leq 2\exp\left(-\frac{2M_k^2\sigma_k^2}{M_k/(1-\gamma)^2}\right)\\
& =2\exp\left(-2(1-\gamma)^2M_k\sigma_k^2\right).
\end{align*}
Applying the union bound across all $(s,a)\in\cS\times\cA$, we obtain
\begin{equation}\label{eqn:Q-union-bd}
\Prob\left(\left\|\hQ(\pik)-\bQ(\pik)\right\|_\infty\geq\sigma_k\right) 
\leq 2\dS\dA \exp\left(-2(1-\gamma)^2M_k\sigma_k^2\right).
\end{equation}
Therefore, for any $\delta\in(0,1)$, if we choose $M_k$ large enough, i.e.,
\begin{equation}\label{eqn:hp-Mk-bd}
M_k \geq \frac{1}{2(1-\gamma)^2\sigma_k^2}\,\log\left(\frac{2\dS\dA}{\delta}\right),
\end{equation}
then $\bigl\|\hQ(\pik)-\bQ(\pik)\bigr\|_\infty<\sigma_k$
with probability at least $1-\delta$.
Combining with~\eqref{eqn:hp-bias-bd} and~\eqref{eqn:hp-error-decomp},
we conclude that 
with probability at least $1-\delta$,
\[
\left\|\hQ(\pik)-Q(\pik)\right\|_\infty 
\leq \frac{\gamma^{H}}{1-\gamma} + \sigma_k.
\]
Finally setting $\sigma_k=\gamma^{H}/(1-\gamma)$ gives the desired result.
\end{proof}

The next theorem characterizes the sample complexity of the inexact PMD method with the simple $Q$-estimator.

\begin{theorem}\label{thm:sample-high-prob}
Consider using the $Q$-estimator~\eqref{eqn:Q-hat} in the inexact PMD method~\eqref{eqn:inexact-pmd-s}, with the step sizes satisfying 
$\eta_0\geq\frac{1-\gamma}{\gamma}\Dzero$ 
and $\eta_{k+1}\geq\frac{1}{\gamma}\eta_k$ for all $k\geq 0$.
For any $\delta\in(0,1)$ and integers $H>0$ and $K>0$, suppose the batch sizes $M_k$ satisfy
\begin{equation}\label{eqn:Mk-high-prob}
M_k \geq \frac{\gamma^{-2H}}{2}\,\log\left(\frac{2K\dS\dA}{\delta}\right),
\qquad k=0,1,\ldots,K-1.
\end{equation}
Then we have with probability at least $1-\delta$,
\begin{equation}\label{eqn:inexact-high-prob}
\V_\rho(\piK) - \V_\rho^\star 
\leq \biggl(1-\frac{1}{\vartheta_\rho}\biggr)^K \frac{2}{1-\gamma}
+ \frac{8\vartheta_\rho}{(1-\gamma)^2}\gamma^H.
\end{equation}
In addition, for any $\epsilon>0$, we have $\V_\rho(\piK) - \V_\rho^\star \leq\epsilon$ with probability at least $1-\delta$ if
\begin{equation}\label{eqn:K-H}
K\geq \vartheta_\rho\log\frac{4}{(1-\gamma)\epsilon}
\qquad\mbox{and}\qquad
H\geq \frac{1}{1-\gamma}\log\frac{16\vartheta_\rho}{(1-\gamma)^2\epsilon}.
\end{equation}
The corresponding sample complexity of state-action pairs is
\begin{equation}\label{eqn:hp-total-samples}
\widetilde{O}\left(\frac{\dS\dA}{(1-\gamma)^8\epsilon^2}\left\|\frac{d_\rho(\pistar)}{\rho}\right\|_\infty^3\right),
\end{equation}
where the notation $\widetilde{O}(\cdot)$ hides poly-logarithmic factors of $1/(1-\gamma)$, $1/\epsilon$ and $\dS\dA/\delta$. 
\end{theorem}

\begin{proof}
Suppose the total number of iterations is~$K$. 
In order to have~\eqref{eqn:Q-high-prob-bd} hold for all $k=0,1,\ldots,K-1$, we need to apply the union bound across all~$K$ iterations, which imposes an additional factor~$K$ on the right-hand side of~\eqref{eqn:Q-union-bd}. Consequently, we can extend Lemma~\ref{lem:Q-high-prob-bd} to ensure that the event
\[
\left\|\hQ(\pik)-Q(\pik)\right\|_\infty 
\leq \frac{2\gamma^H}{1-\gamma}, 
\qquad k=0,1,\ldots,K-1,
\]
occurs with probability at least $1-\delta$ provided that~\eqref{eqn:Mk-high-prob} holds.
Then~\eqref{eqn:inexact-high-prob} follows directly from Theorem~\ref{thm:inexact-linear-rate} with $\tau=2\gamma^H/(1-\gamma)$.

In order to have $\V_\rho(\pik)-\V_\rho^\star\leq\epsilon$ within $K$ iterations, it suffices to have each of the two terms on the right-hand side of~\eqref{eqn:inexact-high-prob} less than $\epsilon/2$, i.e., 
\[
\biggl(1-\frac{1}{\vartheta_\rho}\biggr)^K \frac{2}{1-\gamma} \leq\frac{\epsilon}{2} \qquad\mbox{and}\qquad
\frac{8\vartheta_\rho}{(1-\gamma)^2}\gamma^H \leq \frac{\epsilon}{2}.
\]
which translate into the conditions on~$K$ and~$H$ in~\eqref{eqn:K-H}.
Correspondingly, the batch sizes need to satisfy $M_k\geq M$ where
\[
M:=~\frac{\gamma^{-2H}}{2}\,\log\left(\frac{2K\dS\dA}{\delta}\right)
~\geq~
\frac{1}{2}\left(\frac{16\vartheta_\rho}{(1-\gamma)^2\epsilon}\right)^2 \log\left(\frac{2K\dS\dA}{\delta}\right).
\]
The total number of state-action samples can be estimated as
\begin{align*}
&\dS\dA \cdot K \cdot H \cdot M  \\
=~&\dS\dA \cdot \vartheta_\rho\log\left(\frac{4}{(1-\gamma)\epsilon}\right)
\cdot\frac{1}{1-\gamma}\log\left(\frac{16\vartheta_\rho}{(1-\gamma)^2\epsilon}\right) \cdot\frac{1}{2}\left(\frac{16\vartheta_\rho}{(1-\gamma)^2\epsilon}\right)^2 \log\left(\frac{2K\dS\dA}{\delta}\right) \\
=~&\frac{128\dS\dA\vartheta_\rho^3}{(1-\gamma)^5\epsilon^2}\log\left(\frac{4}{(1-\gamma)\epsilon}\right)\log\left(\frac{16\vartheta_\rho}{(1-\gamma)^2\epsilon}\right) \log\left(\frac{2K\dS\dA}{\delta}\right) \\
=~&\widetilde{O}\left(\frac{\dS\dA\vartheta_\rho^3}{(1-\gamma)^5\epsilon^2}\right).
\end{align*}
Finally, plugging in the definition $\vartheta_\rho=\frac{1}{1-\gamma}\left\|\frac{d_\rho(\pistar)}{\rho}\right\|_\infty$ gives the estimate in~\eqref{eqn:hp-total-samples}.
\end{proof}
\vspace{-1ex}

The sample complexity obtained in Theorems~\ref{thm:sample-high-prob} has $O(\epsilon^{-2})$ dependence on~$\epsilon$.
This is better than that of $O(\epsilon^{-4})$ obtained by \citet{Shani2020aaai} and \citet{Agarwal2021jmlr} and $O(\epsilon^{-3})$ by \citet{LiuZhangBasarYin2020} for policy gradient type of methods (without regularization).
\citet{Cen2020NPGentropy} remarked that $O(\epsilon^{-2})$ sample complexity can be obtained with entropy regularization.
Their approach leads to a result without the factor of the distribution mismatch coefficient, but with the same $1/(1-\gamma)^8$ factor.
\citet{Lazaric2016} derived an $O(\epsilon^{-2})$ sample complexity for a variant of the policy iteration method, with a factor of at least $1/(1-\gamma)^7$.
\citet{Lan2021pmd} studies sample complexity
in expectation instead of with high probability and 
obtains similar results with weaker dependence on $1/(1-\gamma)$.
\citet{YuanGowerLazaric2021} characterize the sample complexity of vanilla policy gradient method (such as REINFORCE \citep{Williams1992}) under a variety of different assumptions on the parametrized value function.

Much progresses have been made for understanding the sample complexity of DMDP in the tabular setting. 
\citet{AzarMunosKappen2013} established a lower bound of 
$\widetilde\Omega\left(\frac{\dS\dA}{(1-\gamma)^3\epsilon^2}\right)$ 
for DMDP under a \emph{generative model}, which allows drawing random state-transitions repeatedly under any policy.
The simple Q-estimator we use in this section fits this sample oracle model, but the dependence of our results on $1/(1-\gamma)$ is much worse than the lower bound. 
On the other hand, this lower bound has been matched or nearly matched by several recent work based on variance-reduced Value Iteration \citep{Sidford2018nearoptimal} and $Q$-learning \citep{Wainwright2019Q}. 
There are interesting work to be done for improving the sample complexity of stochastic policy gradient methods.

\section{Conclusion and Discussion}
\label{sec:conclusion}

We developed a general theory of weak gradient-mapping dominance and used it to obtain an improved sublinear convergence rate of the projected policy gradient methods. 
By exploiting additional structure of discounted Markov decision problem (DMDP), we show that with a simple, non-adaptive rule of geometrically increasing the step sizes, policy mirror descent methods enjoy linear convergence without relying on entropy or other strongly convex regularizations. 
%
%
In fact, the convergence rates obtained with strongly convex regularizations \citep{Cen2020NPGentropy,Lan2021pmd,Zhan2021Regularized} are no better than $\gamma^k$ regardless of the regularization strength.

Our results on policy mirror descent methods show that dynamic preconditioning using discounted state-visitation distributions
is critical for obtaining fast convergence rates that are (almost) independent of problem dimensions.
The adopted local Bregman divergence, being KL-divergence or squared Euclidean distance, does not make much difference.
Indeed, when the step sizes grow to infinity, preconditioned policy mirror descent methods derived with different Bregman divergences all reduce to the classical Policy Iteration algorithm.
Essentially, such methods with finite step sizes can be viewed as inexact Policy Iteration methods, much like many approximate dynamic programming algorithms.

The major limitation of this work is our restriction to direct policy parametrization.
(We note that the NPG method with tabular softmax parametrization has an equivalent mirror-descent form expressed in the policy space, therefore is included in our study.)
A natural extension is to consider general policy parametrizations of the form $\pi(\theta)$ where the dimension of~$\theta$ is much smaller than $\dS\dA$.
There are two ways to proceed.
The first approach is to simply treat it as a nonlinear optimization problem of minimizing the composite objective $J_\rho(\theta)=\V_\rho(\pi(\theta))$.
This approach may lose some important structure of DMDP.
In particular, the parametrized objective function $J_\rho(\theta)$ may no longer be quasi-convex or quasi-concave. As a result, it will be hard to establish convergence to global optimum and we may have to rely on standard theory of smooth nonconvex optimization, which imposes bounded step sizes and leads to relatively slow convergence rates.

The second approach is to follow the framework of \emph{compatible function approximation} \citep{Sutton2000PolicyGrad,Kakade2001NPG}, which is extensively developed by \citet{Agarwal2021jmlr}.
This approach facilitates the extension of our results on inexact policy mirror descent
to general policy parametrization.
In particular, our results in Section~\ref{sec:inexact-pmd} show that geometrically increasing step sizes do not cause instability even if the $Q$-functions are evaluated inaccurately. 
In fact, inexact policy mirror descent methods converge linearly up to an asymptotic error floor,
which immediately leads to an $O(\epsilon^{-2})$ sample complexity as we have shown. 
It is of great interest to reduce the dependence of sample complexity on $1/(1-\gamma)$ and the distribution mismatch coefficient.

\acks{%
The author is grateful to Lihong Li and Simon S.\ Du for helpful discussions and feedback. 
Parts of the results in this paper were obtained by the author while preparing for a tutorial jointly with Lihong Li at the SIAM Conference on Optimization held in July 2021.

The author is indebted to Marek Petrik and Julien Grand-Clement, who found a mistake in a previous version of this paper stating that the weighted value function is quasi-convex and quasi-concave. They gave a simple counter-example and pointed out the mistake in the proof. 
Indeed, it is neither quasi-convex nor quasi-concave.
Fortunately this mistake does not affect the rest of the results that are contained in this version.
}

\appendix
\section{Appendix}

\subsection{Derivation of Policy Gradient using Matrix Calculus}
\label{sec:apdx:policy-grad}

We derive the policy gradient formula~\eqref{eqn:policy-grad-s} using simple matrix calculus.
Let $e_s\in\R^\dS$ be a vector  with components $e_{s,s'}=1$ if $s=s'$ and $0$ otherwise.
From the expression of $\V(\pi)$ in~\eqref{eqn:v-def-P-r}, we can write its components as
\begin{align*}
\V_s(\pi) =e_s^T \V(\pi) = e_s^T \bigl(I-\gamma P(\pi)\bigr)^{-1}r(\pi).
\end{align*}
Using the matrix calculus formula 
$\frac{\partial X^{-1}}{\partial \pi} = -X^{-1}\frac{\partial X}{\partial \pi} X^{-1}$
with $X=(I-\gamma P(\pi))$, we have
\begin{align*}
\frac{\partial \V_s(\pi)}{\partial \pi_{s',a'}}
&=e_s^T\left(-\bigl(I\!-\!\gamma P(\pi)\bigr)^{-1} \left(-\gamma\frac{\partial P(\pi)}{\pi_{s',a'}}\right) \bigl(I\!-\!\gamma P(\pi)\bigr)^{-1}\right) r(\pi) + e_s^T\bigl(I\!-\!\gamma P(\pi)\bigr)^{-1} \frac{\partial r(\pi)}{\partial \pi_{s',a'}}\\
&=e_s^T\bigl(I-\gamma P(\pi)\bigr)^{-1} \left( 
\frac{\partial r(\pi)}{\partial \pi_{s',a'}} 
+ \gamma\frac{\partial P(\pi)}{\pi_{s',a'}} 
\bigl(I-\gamma P(\pi)\bigr)^{-1} r(\pi) \right)\\
&=e_s^T\bigl(I-\gamma P(\pi)\bigr)^{-1} \left( R_{s',a'} e_{s'} + \gamma \frac{\partial P(\pi)}{\pi_{s',a'}} \V(\pi) \right) 
\end{align*}
where in the last equality we used $\partial r(\pi)/\partial\pi_{s',a'}=R_{s',a'}e_{s'}$ and the definition of $\V(\pi)$.
From the definition of $P(\pi)$, we have $\partial P(\pi)/\partial \pi_{s',a'} = e_{s'}P(\cdot|s',a')$, which is a rank-one matrix with $P(\cdot|s',a')$ acting as a row vector. Therefore,
\begin{align*}
\frac{\partial \V_s(\pi)}{\partial \pi_{s',a'}}
&=e_s^T\bigl(I-\gamma P(\pi)\bigr)^{-1} e_{s'} \Bigl( R_{s',a'} + \gamma P(\cdot|s',a') \V(\pi) \Bigr) 
 = \frac{1}{1-\gamma} d_{s,s'}(\pi) Q_{s',a'}(\pi),
\end{align*}
where we used the expression of $d_{s,s'}(\pi)$ in~\eqref{eqn:dsv-matrix} and the definition of $Q_{s',a'}(\pi)$.
This gives the component-wise expression for policy gradient, which leads to the aggregated form~\eqref{eqn:policy-grad-s}.

\subsection{Strong Gradient-Mapping Domination}
\label{sec:apdx:grad-map-dom}

Following the setting in Section~\ref{sec:grad-map-dom}, we define a stronger notion of gradient-mapping domination and show that it leads to geometric convergence to a global optimum.

\begin{definition}[\textbf{strong gradient-mapping domination}]
\label{def:strong-grad-dom}
Suppose $F:=f+\Psi$ where $f$ is $L$-smooth and $\Psi$ is proper, convex and closed. We say that~$F$ satisfies a \emph{strong gradient-mapping dominance} condition if there exists $\mu>0$ such that 
\begin{equation}\label{eqn:strong-grad-map-dom}
\frac{1}{2}\left\|G_L(x)\right\|_2^2 \geq\mu\bigl(F(T_L(x))-F^\star\bigr),
\qquad \forall\,x\in\dom\Psi,
\end{equation}
where $F^\star=\min_x F(x)$ and $T_L$ and $G_L$ are defined in~\eqref{eqn:T-def} and~\eqref{eqn:G-def} respectively.
\end{definition}

Consider the composite optimization problem of minimizing $F:=f+\Psi$ where $f$ is $L$-smooth and $\Psi$ is proper, convex and closed. 
If~$F$ satisfies the strong gradient-mapping domination condition, then the proximal gradient method~\eqref{eqn:prox-grad-method} converges geometrically to a global minimum. 
To see this, we simply combine the descent property~\eqref{eqn:grad-map-descent} with strong gradient-mapping dominance condition~\eqref{eqn:strong-grad-map-dom} to obtain
\[
F(x^k) - F(x^{k+1}) ~\geq~ \frac{1}{2L} \left\|G_L(x^k)\right\|_2^2
~\geq~ \frac{\mu}{L} \bigl(F(x^{k+1})-F^\star\bigr).
\]
Rearranging terms, we obtain
\[
\left(1+\frac{\mu}{L}\right) \bigl(F(x^{k+1}) - F^\star \bigr) \leq F(x^k)-F^\star.
\]
This leads to a geometric recursion and we have 
\[
F(x^{k}) - F^\star 
\leq \left(1+\frac{\mu}{L}\right)^{-k} \bigl(F(x^0)-F^\star\bigr) .
\]

\paragraph{Connections with other notions of gradient dominance.}
The classical Kurdyka-{\L}ojasiewicz (K{\L}) condition with exponent~$1/2$ \citep{Kurdyka1998} can be expressed as
\begin{equation}\label{eqn:KL-condition}
\min_{s\in\partial F(x)} \frac{1}{2}\|s\|^2 \geq \tilde\mu\bigl(F(x)-F^\star\bigr),
\qquad \forall\,x\in\dom\Psi,
\end{equation}
where $\partial F(x)$ denotes the set of subgradients (subdifferential) of~$F$ at~$x$.
\citet{Karimi2016PL} derived a proximal Polyak-{\L}ojasiewicz (P{\L}) condition
\begin{equation}\label{eqn:PL-condition}
\frac{1}{2}P_L(x) \geq \mu\bigl(F(x)-F^\star\bigr),
\qquad \forall\,x\in\dom\Psi,
\end{equation}
where
\[
P_L(x):= -2L\min_y\left\{\langle\nabla f(x), y-x\rangle + \frac{L}{2}\|y-x\|_2^2 + \Psi(y)-\Psi(x)\right\},
\]
and showed that it is equivalent to the K{\L} condition~\eqref{eqn:KL-condition} in the sense that they imply each other albeit with different constants~$\tilde{\mu}$ and~$\mu$.
Interestingly, it can be shown that there is an interlacing relationship between our definition of gradient-mapping domination and the proximal P{\L} condition:
\[
\frac{1}{2}\|P_L(x)\|_2^2
~\geq~ \frac{1}{2}\left\|G_L(x)\right\|_2^2
~\geq~ \mu \bigl(F(x)-F^\star\bigr)
~\geq~ \mu \bigl(F(x^+)-F^\star\bigr).
\]
The proximal P{\L} condition~\eqref{eqn:PL-condition} takes the first and the third terms in the above inequality chain, while our gradient-mapping dominance condition~\eqref{eqn:strong-grad-map-dom} takes the second and the fourth. 
We conjecture that these two conditions also imply each other.

\bigskip

\bibliography{policy_grad}

\end{document}